\documentclass[12pt]{article}
\usepackage{amssymb}

\usepackage[dvipdfm,bookmarks=true,bookmarksnumbered=true,bookmarkstype=toc]{hyperref}   
\usepackage{amsmath, amsthm, amsfonts}
\usepackage{amscd}
\usepackage{verbatim}

\numberwithin{equation}{section}

\numberwithin{figure}{section}

\newtheorem{theorem}{Theorem}[section]
\newtheorem{corollary}{Corollary}[section]
\newtheorem{lemma}{Lemma}[section]
\newtheorem{proposition}{Proposition}[section]
\newtheorem{definition}{Definition}[section]
\newtheorem{example}{Example}[section]
\newtheorem{remark}{Remark}[section]
\newtheorem{remarks}{Remarks}[section]
\newtheorem{conjecture}{Conjecture}[section]

\newcommand{\be}{\begin{equation}}

\newcommand{\core}{C_0^{\infty}(\Omega)}
\newcommand{\laplace}{\Delta}
\newcommand{\pl}{\laplace_p}
\newcommand{\grad}{\nabla}
\newcommand{\pd}{\partial}
\newcommand{\bo}{\pd}

\newcommand{\sm}{\setminus}

\newcommand{\diver}{\mathrm{div}\,}

\newcommand{\resc}{\mathcal{D}}
\newcommand{\fundot}{\circ}
\newcommand{\beq}{\begin{equation}}
\newcommand{\eeq}{\end{equation}}
\newcommand{\beqa}{\begin{eqnarray}}
\newcommand{\eeqa}{\end{eqnarray}}
\newcommand{\beqanl}{\begin{eqnarray*}}
\newcommand{\eeqanl}{\end{eqnarray*}}
\newcommand{\bs}{\begin{sub}}
\newcommand{\es}{\end{sub}}
\newcommand{\bsn}{\begin{subn}}
\newcommand{\esn}{\end{subn}}
\newcommand{\bea}{\begin{eqnarray}}
\newcommand{\eea}{\end{eqnarray}}
\newcommand{\bean}{\begin{eqnarray*}}
\newcommand{\eean}{\end{eqnarray*}}
\newcommand{\BA}[1]{\begin{array}{#1}}
\newcommand{\EA}{\end{array}}

\newlength{\wex}  \newlength{\hex}
                \def\gz{\zeta}

\def\squarebox#1{\hbox to #1{\hfill\vbox to #1{\vfill}}}

\newcommand{\lsim}{\precsim}
\newcommand{\gsim}{\succsim}

\begin{document}
\title{\Large Positive Liouville theorems and asymptotic behavior for $p$-Laplacian type elliptic equations with a Fuchsian potential}
\author{\large Martin Fraas \\
Department of Physics, Technion, Haifa, Israel \\ martin.fraas@gmail.com
\\[4mm]
Yehuda Pinchover \\
Department of Mathematics, Technion, Haifa, Israel \\ pincho@techunix.technion.ac.il}


\maketitle

{\em \begin{center} Dedicated with affection and deep esteem\\ to the memory of Michelle Schatzman. \end{center}}
\vskip 2mm
\abstract{
We study positive Liouville theorems and the asymptotic behavior of positive solutions of $p$-Laplacian type elliptic equations of the form
 \begin{equation*}
Q'(u):=-\pl (u) + V |u|^{p-2} u = 0 \qquad \mbox{ in } X,
\end{equation*}
where $X$ is a domain in $\mathbb{R}^d$, $d\geq 2$, and $1<p<\infty$. We assume that the potential $V$ has a Fuchsian type singularity at a point $\zeta$, where either $\zeta=\infty$ and $X$ is a truncated $C^2$-cone, or $\zeta=0$ and $\zeta$ is either an isolated point of $\pd X$ or belongs to a $C^2$-portion of $\pd X$.
\\[4mm]
}

\section{Introduction}\label{sect1}

A function $u$ is called {\bf $p\,$-harmonic} in a domain $X\subset \mathbb{R}^d$ if
$$-\pl(u)=0 \qquad \mbox{ in } X. $$
Here $\pl(u):=\diver(|\nabla u|^{p-2}\nabla u)$ is the celebrated $p\,$-Laplacian.

The positive Liouville theorems for $p\,$-harmonic functions  states that a nonnegative entire $p\,$-harmonic function on $\mathbb{R}^d$ is constant (see for example \cite{Martio}). On the other hand, Riemann's removable singularity theorem for $p\,$-harmonic functions with $p\leq d$ claims \cite{KV} that if $u$ is a positive $p\,$-harmonic function in the punctured unit ball
$B_1\sm \{0\}$, then either $u$ has a removable singularity at the origin, or
  \begin{equation*}
 u(x) \;\underset{x\to 0}{\sim}\;
\left\{
  \begin{array}{ll}
   \quad |x|^\frac{p-d}{p-1} & \hbox{ if }\; p<d,\\[2mm]
    -\log |x| & \hbox{ if }\; p=d.
  \end{array}
\right.
\end{equation*}
Furthermore, Picard's principle for $p\,$-harmonic functions claims that up to a multiplicative constant there exists a unique positive $p\,$-harmonic function in the punctured unit ball $B_1\sm \{0\}$ which vanishes on $\pd B_1$ (see for example Theorem~\ref{thm:weakI}).
Finally, Poisson`s principle states that for a given $\gz\in \pd B_1 $, the cone of positive {\em harmonic} functions in the unit ball that vanish on $\pd B_1 \sm \{\gz\}$ is of one dimension.

The aim of this paper is to study positive Liouville theorems, Picard-type principles, and removable singularity theorems for more general equations. More precisely, we study the {\em uniqueness} (up to a multiplicative constant) of certain positive solutions of the quasilinear elliptic equation
\begin{equation}\label{eq:1}
Q'_V(u)=Q'(u):=-\pl (u) + V |u|^{p-2} u = 0 \qquad \mbox{ in }  X,
\end{equation}
where $1<p<\infty$, $X$ is a domain in $\mathbb{R}^d$, $d\geq 2$,
and  $V \in L^\infty_{\mathrm{loc}}(X)$ is a potential with a Fuchsian type singularity at a fixed point $\zeta$ which belongs to the (ideal) boundary of $X$. We also study the asymptotic behavior of the \textit{quotients} of two positive solutions near the singular point $\zeta$.

Unless otherwise stated, we assume throughout the paper that
\begin{equation}
\label{Q} Q(u):=\int_X\left(|\nabla u|^p+V|u|^p\right)\,\mathrm{d}x\geq 0
\qquad u\in C_0^\infty(X).
\end{equation}
In other words (see \cite{ky3}),  {\bf we assume that \eqref{eq:1} admits a positive solution}.

Without loss of generality, we suppose that the singular point $\zeta$ is either the origin (so, $\zeta=0$), or $\zeta = \infty$.
More precisely we consider the following two cases:
\begin{enumerate}
\item $X$ is a domain (which might be unbounded and nonsmooth) such that the singular
point $\zeta=0$ is either an isolated component of the boundary $\pd X$,  or $\zeta=0$ belongs to a $C^2$-portion of $\pd X$.
\item $X$ is a cone near infinity, and  $\zeta = \infty$. More precisely, the intersection of $X$ with the exterior of some ball is an open connected truncated cone with $C^2$ boundary (this boundary might be empty; in this case $X$ is an exterior domain, and $\gz= \infty$ is an isolated singular point).
  \end{enumerate}
\begin{remark}{\em 
The assumption in (1) that $\zeta$ belongs to $C^2$-portion of boundary  should
be considered as a technical condition under which the boundary Harnack
inequality is valid. We expect that our results hold true under
milder smoothness assumptions.
 }
\end{remark}
Since we allow $X$ to be unbounded and $\zeta = \infty$, it is convenient to consider
the one-point compactification $\widehat{\mathbb{R}^d}:=\mathbb{R}^d \cup \{\infty\}$ of $\mathbb{R}^d$. We denote by $\hat{X}$  the closure of $X$ in $\widehat{\mathbb{R}^d}$.
 On the other hand, by a {\bf neighborhood of infinity in $X$} we mean a set of the form $X\sm K$, where  $K\Subset X$ (we write $\Omega_1 \Subset \Omega_2$
if $\Omega_2$ is open, $\overline{\Omega_1}$ is
compact and $\overline{\Omega_1} \subset \Omega_2$).

\begin{definition}\label{def:Fuchs}{\em 
Let $\zeta\in \pd \hat{X}$, where $\zeta\in \{0,\infty\}$.  We say that (\ref{eq:1}) has a {\bf Fuchsian type singularity at $\zeta$} if there exists a relative neighborhood $ X'\subset X\subset \hat{X}$ of $\zeta$ and a positive constant $C$ such that
\begin{equation}\label{eq:FuchsCon}
 |x|^p |V(x)| \leq C \qquad \mbox{for a. e. }  x\in X'.
    \end{equation}
    }
 \end{definition}
 Fuchsian type equations form a natural class where positive Liouville theorems, Picard's principle and removable singularity theorems hold true (see \cite{Borghol,De_Cicco,Lancia,Pin1,PS,V} and the references therein). In particular, under some restrictions, Poisson's principle for a Fuchsian type $p\,$-Laplace equation of the form \eqref{eq:1} in a bounded smooth domain is proved in \cite{Ver}. For other Liouville theorems for quasilinear equations see for example \cite{Birin,Far,Lewis1,Ser},  and the references therein.

 One of the main tools of the present paper is a dilatation process which uses the simple observation that near a Fuchsian-type singularity Eq.~\eqref{eq:1} is quasi-invariant under the scaling $x\mapsto Rx$.
 We postpone the description  of the dilatation process to the next section, and we continue the introduction by stating our main results.

First we recall two types of positive solutions of minimal growth \cite{Agmon,Pin1,Pin2}.
  \begin{definition}\label{def:minimal_gr}{\em 
1. Let $K\Subset X$, and let $u$ be a positive solution of the equation $Q'(u)=0$ in  $X\sm K$. We say that $u$ is {\bf a positive solution of  minimal growth in a neighborhood of infinity in $X$}  if for any $K\Subset K'\Subset X$ with smooth boundary  and any positive supersolution $v\in C((X\sm K')\cup \pd K')$ of the equation $Q'(w)=0$ in $X \sm K'$ satisfying  $u \leq v$ on $\pd K'$, we have $u \leq v$ in $X \sm K'$.

2.  Let $\zeta\in \pd \hat{X}$, and let $u$ be a positive solution of the equation $Q'(u)=0$ in  $X$. We say that $u$ is {\bf a positive solution of  minimal growth in a neighborhood of $\bo \hat{X}\sm \{\zeta\}$} if for any relative neighborhood $K'\Subset \hat{X}$ of $\zeta$ such that  $\Gamma:= \pd K' \cap X$ is smooth, and for any positive supersolution $v\in C((X\sm K')\cup \Gamma)$ of the equation $Q'(w)=0$ in $X \sm K'$ satisfying  $u \leq v$ on $\Gamma$, we have $u \leq v$ in $X \sm K'$.
 }
    \end{definition}
We summarize some basic properties of solutions of minimal growth in a neighborhood of infinity at the end of Section~\ref{sec2} (see Remarks~\ref{rem1}).

Next, we introduce a partial order $\lsim$ on a certain set $\mathcal{G}_\gz$ of {\em germs} at $\gz$. Denote by $\mathcal{G}_\gz$ the set of all positive solutions $u$ of the equation $Q'(w)=0$ in some neighborhood $X' \subset X$ of $\zeta$ that vanish continuously
on $(\pd  X' \cap \pd  X)\sm \{\zeta\}$ ($X'$ might depend on $u$).

Let $u,v\in \mathcal{G}_\gz$. We use the following notations.
\begin{itemize}
\item
We denote $u\underset{x\to \zeta}{\sim} v\;$ if
$\;\displaystyle{\underset{x \in X}{\underset{x \to \zeta}{\lim}}\,\frac{u(x)}{v(x)}= C}$ for some  positive constant $C$.
\item By $u \underset{x \to \zeta}{\prec}v$ we mean that $\displaystyle{\underset{x \in X}{\underset{x \to \zeta}{\lim}}\,\frac{u(x)}{v(x)} =0}$.

 \item By $u \underset{x \to \zeta}{\lsim} v $ we mean that either $u \underset{x \to \zeta}{\sim} v$ or $u \underset{x \to \zeta}{\prec}v$.

\item  We denote $u\underset{x \to \zeta}{\succ} u\;$ if $\;v \underset{x \to \zeta}{\prec}u$. Similarly,  $u\underset{x \to \zeta}{\gsim} v\;$ if $\;v \underset{x \to \zeta}{\lsim} u$.

     \item The dependence on $\gz$ in the above notations will be omitted when there is no danger of confusion.
     \end{itemize}
Clearly, $u {\sim} v$ defines an equivalence relation and equivalence classes on $\mathcal{G}_\gz$.
\begin{definition}\label{def:regular} {\em 
Fix $\zeta \in \pd \hat X$.

1. Let $u\in \mathcal{G}_\gz$.
We say that $\zeta$ is a {\bf regular point with respect to the solution $u$} if for any  $v\in \mathcal{G}_\gz$ we have
$$\mbox{ either }\; u \lsim v \quad \mbox{or } \; u \gsim v.$$

2. We say that $\zeta$ is a {\bf regular point of the equation $Q'(w)=0$ in $X$}   if for any two positive solutions $u,v\in \mathcal{G}_\gz$ we have
$$\mbox{ either }\; u \lsim v \quad \mbox{or } \; u \gsim v.$$

3. By {\bf uniqueness of positive solutions} we always mean uniqueness up to a multiplicative constant.
 }
\end{definition}

\begin{remarks} {\em 
(1) $\zeta$ is a regular point of (\ref{eq:1}) if any $u,v\in \mathcal{G}_\gz$ are comparable with respect to the $\lsim$ ordering. In other words, regularity means that the ordering $\lsim$ is total on $\mathcal{G}_\gz$.

(2) Equivalently, $\zeta$ is a regular point of the equation (\ref{eq:1}) if for
any two solutions $u,v\in \mathcal{G}_\gz$ the limit
$$
\underset{x \in X'}{\underset{x \to \zeta}{\lim}} \frac{u(x)}{v(x)}
\quad \mbox{exists.}
$$
 }
\end{remarks}

 Now we are ready to state our main conjecture which is partially answered in the present paper.
   \begin{conjecture}\label{main_conj}
Assume that Eq.~\eqref{eq:1} has a Fuchsian type singularity at $\zeta\in \pd \hat X$ and admits a (global) positive solution. Then
\begin{enumerate}
\item[i)] $\zeta$ is a regular point of equation (\ref{eq:1}).
\item[ii)]  Equation \eqref{eq:1} admits a unique (global) positive solution of minimal growth in a neighborhood of $\bo \hat{X} \sm \{\zeta\}$.
\end{enumerate}
   \end{conjecture}
\begin{remarks}\label{rem:Liouville}{\em 
(1) Part i) of Conjecture \ref{main_conj} should be considered as a removable singularity statement, while part ii) is a Liouville (or Picard)-type  statement. In particular,  if  $X = \mathbb{R}^d$ and $\zeta = \infty$, then Conjecture \ref{main_conj} ii) asserts that the positive Liouville theorem is valid.

(2) One can rephrase part ii) of  Conjecture \ref{main_conj} as saying that the Martin boundary of
\eqref{eq:1} at a Fuchsian type singular point is a singleton.
 }
\end{remarks}
It turns out that the first part of Conjecture~\ref{main_conj} implies the second part. We have:
\begin{proposition} \label{thm:RegUni}
Assume that equation (\ref{eq:1}) has a Fuchsian type singularity at $\zeta\in \pd \hat X$ and admits a (global) positive solution. If $\zeta$ is
a regular point of \eqref{eq:1}, then (\ref{eq:1}) admits a unique positive
solution in $X$ of minimal growth in a neighborhood of $\pd \hat X \sm \{\zeta\}$.
\end{proposition}
   \begin{remarks}\label{rem:1} {\em 
    (1)\! Conjecture \ref{main_conj} holds true for second-order {\em linear} (not necessarily symmetric) equations with a Fuchsian type singularity at $\zeta$. Moreover, in this case, if $\gz$  is a nonisolated singular point it is sufficient to assume that $\pd X$ is Lipschitzian near $\zeta$.
In particular, the conjecture holds for Eq.~\eqref{eq:1} with $p=2$ \cite{Pin1}.

(2)
Let $\zeta\in \pd \hat{X}$ be a Fuchsian isolated singular point, and let $u$ be a positive solution of the equation $Q'(w)=0$ in some neighborhood $X' \subset X$ of $\zeta$. It was shown in examples~9.1 and 9.2 of \cite{Pin1} that
$$
\underset{x\in X'}{\underset{x \to \zeta }{\lim}} u(x)
$$
might not exist even in the linear case ($p=2$). Nevertheless, in Theorem~\ref{thm:nonnegativeV} we partially answer Question~9.5 of \cite{Pin1} by proving that (under some further assumptions) if $V$ is a {\em nonnegative} potential and $\zeta$ is a Fuchsian isolated singular point of  $\pd \hat{X}$, then
$$
\underset{x\in X'}{\underset{x \to \zeta }{\lim}} u(x)  \quad \mbox{exists}.
$$
In a subsequent paper  \cite{MY3},  the authors prove the existence of the limit for sign changing $V$ satisfying additional assumptions.
 }
   \end{remarks}
We first concentrate on the case where $\zeta$ is an {\em isolated point}  of
the boundary $\pd \hat{X}$, that is, either $\zeta = 0$ and $X$ is a punctured neighborhood of
the origin, or $\zeta = \infty$ and $X$ is an exterior domain. We postpone
the study of non-isolated singularity to Section \ref{sec:boundary}.
 For isolated singularities we prove  that the two statements of Conjecture \ref{main_conj} hold true in two particular cases. In the first case we deal with {\em weakly Fuchsian singular point} (to be defined later), where we strengthen our assumption on the behavior of the potential $V$ near the (isolated) singular point $\zeta$.  For the precise formulation of this result see Theorem~\ref{thm:weak}. Meanwhile, for illustration we present a particular case of the aforementioned theorem.
\begin{theorem}\label{thm:weakI}
 Suppose that (\ref{eq:1}) admits a positive solution, and let $\zeta \in \pd \hat{X}$ be an isolated point of $\pd \hat{X}$.  Assume that $|x|^p V(x)$
is continuous near $\zeta$, and that
$$
\lim_{x \to \zeta} |x|^p V(x) = 0.
$$
  Then the two assertions of Conjecture~\ref{main_conj} hold true.
\end{theorem}

Let $X$ be a domain in $\mathbb{R}^d$, $x_0\in X$, and let $V\in L^\infty_{\mathrm{loc}}(X)$.  Without
loss of generality, we may assume that $x_0=0$. By applying Theorem~\ref{thm:weakI} in $X\sm \{0\}$, and $\zeta=0$, we readily obtain the following result which is new for $p>d$ (the case $p \leq d$ follows also from \cite{ky3,Pin2}).
\begin{corollary}\label{cor:dlessp}
Suppose that \eqref{eq:1} admits a positive solution, and let $x_0\in X$. Then the equation $Q'(u)=0$ admits a unique positive solution in $X\sm \{x_0\}$ of  minimal growth in a neighborhood of infinity in $X$.
\end{corollary}

A second case where the assertions of Conjecture~\ref{main_conj} hold true is the {\em spherical symmetric} case. In particular, $\gz$ is again an isolated singularity, and $X$ is one of the following domains: $\mathbb{R}^d$,\,$\mathbb{R}^d\sm\{0\}$,\, $\mathbb{R}^d\sm B_R$,\, $B_R\sm\{0\} $. We have
\begin{theorem}\label{thm:sym}
 Suppose that the domain $X$ and the potential $V$ are spherical
symmetric, and that Eq.~\eqref{eq:1} admits a positive
solution. Assume further that $V$ has a Fuchsian type singularity at
$\zeta$, where $\zeta=0$ or $\zeta=\infty$.  Then

(i) $\zeta$ is a regular point of \eqref{eq:1}.

(ii) Equation \eqref{eq:1} admits a unique positive solution of minimal growth in a
neighborhood of infinity in $X \sm \{\zeta\}$. This solution is spherically
symmetric.

(iii) For any $u\in \mathcal{G} _\gz$ there exists a radial solution $\tilde{u}\in \mathcal{G} _\gz$ such that $u\sim \tilde{u}$.
\end{theorem}

\begin{example}\label{example1}{\em 
Let $X=\mathbb{R}^d\sm \{0\},\,d > 1$. Consider the equation
   \begin{equation}\label{eq:hardy}
    Q'(u):=-\Delta_p(u)-\lambda \frac{|u|^{p-2}u}{|x|^p}=0 \qquad \mbox{ in } X.
   \end{equation}
So, \eqref{eq:hardy} has Fuchsian type singularities at $\zeta =0$ and at $\zeta =\infty$.
By Hardy's inequality  \eqref{eq:hardy} admits a positive solution if and only if  $$\lambda \leq c_H:=\left|\frac{p-d}{p}\right|^p.$$
Moreover, in this case, \eqref{eq:hardy} is critical in $X$ if and only if $\lambda = c_H$ (for the definition of criticality see Remarks~\ref{rem1} (1)).  Furthermore, for  $\lambda = c_H$ the corresponding unique positive (super) solution (ground state) of \eqref{eq:hardy} is given by $u(x)=|x|^{\gamma_*}$, where $\gamma_*:=(p-d)/p$ \cite{ky7}.

On the other hand, if $\lambda < c_H$, then the corresponding radial equation
$$-|v'|^{p-2}\left[(p-1)v''+\frac{d-1}{r}v'\right]-\lambda\frac{|v|^{p-2}v}{r^p}
=0\qquad r\in (0,\infty)$$ has two positive solutions of the form $v_\pm(r):=|r|^{\gamma_\pm(\lambda)}$, where
$\gamma_-(\lambda)<\gamma_*<\gamma_+(\lambda)$, and $\gamma_\pm(\lambda)$ are solutions of the transcendental 	 equation
$$-\gamma |\gamma|^{p-2}[\gamma(p-1)+d-p]=\lambda.$$
Consequently, \eqref{eq:hardy} has two positive solutions of the form $u_\pm(x):=|x|^{\gamma_\pm(\lambda)}$. It follows from Proposition~C.1 in \cite{LLM} (see also \cite{Pin2}, Theorem~7.1) that $u_-$ is a positive solution of minimal growth in  $\pd\hat{X}\sm \{0\}$. By a similar argument, $u_+$ is a positive solution of minimal growth in $\pd\hat{X}\sm \{\infty\}$.
Therefore, Theorem~\ref{thm:sym} implies that
 $u_-$ is the unique positive solution
of minimal growth in $\pd\hat{X}\sm \{0\}$, while $u_+$ is the unique positive solution
of minimal growth in $\pd\hat{X}\sm \{\infty\}$.

 Moreover, we have the following isotropy result. Let $v$ be any positive solution
of the equation $Q'(u)=0$ in a neighborhood of $\gz$. Then the two limits
$$
\lim_{|x| \to \gz} \frac{v(x)}{|x|^{\gamma_\pm}}
$$
exist.  Furthermore, For any $u\in \mathcal{G} _\gz$ there exists a radial solution $\tilde{u}\in \mathcal{G} _\gz$ such that $u\sim \tilde{u}$.
 }
   \end{example}


The outline of the paper is as follows. In Section~\ref{sec2} we introduce the notion of a weak Fuchsian singularity and discuss some other notions and results we need throughout the paper. In particular, we outline the dilatation technique that are used to prove the regularity of a singular point.
The proofs of the main results of the present paper rely on comparison techniques, dilatation arguments, and the regularity of singular points of limiting equations; these issues are discussed in sections
 \ref{sec:comparison}, \ref{sec:Harnack} and Appendix~\ref{sec:asymp}, respectively.
In particular, the key Proposition~\ref{thm:RegUni} is proved in Section
\ref{sec:Harnack}, the proof of Theorem~\ref{thm:sym} appears in
Section \ref{sec:proofs}, and in the appendix we prove for the case $p> d$ the exact asymptotic of positive $p\,$-harmonic functions defined in a neighborhood  of infinity. In Section~\ref{sec:boundary} we extend the results to the case of a nonisolated singularity. Finally, we conclude the paper in Section~\ref{sec7} with some examples, remarks and applications.

   \section{Preliminaries and main results}\label{sec2}
In this section we discuss the necessary background for our study of Liouville theorems and  present the main results of the paper.

The following notations and conventions will be used. We denote by $B_R(x_0)$ (respectively, $S_R(x_0)$) the open ball (respectively, sphere) of radius $R$ and a center at $x_0$, and let $B_R:=B_R(0)$ and $S_R:=S_R(0)$. The exterior of a ball will be denoted by $B_R^*:=\mathbb{R}^d\sm \overline{B_R}$. For a domain $\Omega\subset \mathbb{R}^d$ and $R>0$, we denote $$\Omega/R:=\{x\in \mathbb{R}^d \mid x =R^{-1}y, \;\;\mbox{ where } y\in \Omega \}.$$ Let $f,g \in C(\Omega)$ be positive functions. We use the notation $f\asymp g$ on $\Omega$ if there exists a positive constant $C$ such that
$$C^{-1}g(x)\leq f(x) \leq Cg(x) \qquad \mbox{ for all } x\in \Omega.$$
 We also denote $f_\pm(x):=\max\{0, \pm f(x)\}$, so, $f=f_+-f_-$.

Let $\Omega$ be a subdomain of $X$. By a {\em (weak) solution} of the equation $Q'(u)=0$ in $\Omega$,  we mean
a function
$v\in W^{1,p}_{\mathrm{loc}}(\Omega)$  such that
 \begin{equation} \label{solution} \int_\Omega (|\nabla v|^{p-2}\nabla
v\cdot\nabla\varphi+V|v|^{p-2}v\varphi)\,\mathrm{d}x=0 \qquad \forall \varphi\in\core.
\end{equation}
We say that a real function $v\in C^1_{\mathrm{loc}}(\Omega)$ is a
{\it  supersolution} (respectively, {\it  subsolution}) of
the equation $Q'(u)=0$ in $\Omega$ if for every nonnegative $\varphi\in\core$ we have
 \begin{equation}\label{supersolution}
\int_\Omega (|\nabla v|^{p-2}\nabla
v\cdot\nabla\varphi+V|v|^{p-2}v\varphi)\,\mathrm{d}x\geq 0 \mbox{ (respectively, }\leq 0\mbox{).}
\end{equation}

Next we introduce a dilatation process that uses the quasi-invariance of our equation under the scaling $x\mapsto Rx$.
For $R>0$ let $V_R$ be the {\bf scaled potential} defined by
\begin{equation} \label{eq:wfuchs}
 V_R(x) := R^p V(R x) \qquad  x \in X/R.
 \end{equation}
Let $\{R_n\}\subset \mathbb{R}_+$ be a sequence satisfying   $R_n \to \zeta$
(where $\zeta$ is either $0$ or $\infty$)
such that \begin{equation}\label{eq:rescalepotential}
V_{R_n}\underset{n\to \infty}{\longrightarrow} W \quad \mbox{ in  the weak$^*$ topology of  $L_{\mathrm{loc}}^{\infty}(Y)$,}
\end{equation}
 where $Y:=\lim_{n\to\infty} X/R_n$.
Define the {\bf limiting dilated equation with respect to} (\ref{eq:1}) (and the sequence $\{R_n\}$) as
\begin{equation}
\resc^{\{R_n\}}(Q)(w) := -\pl(w) + W |w|^{p-2} w=0 \qquad \mbox{ on } Y.
\end{equation}
\begin{remark}{\em 
In general $Y$ is a cone.
However, if $\zeta$ is an isolated singular point, then clearly $Y=\mathbb{R}^d\sm\{0\}$.
 }
\end{remark}

The next two propositions give basic properties of limiting dilated equations. The first proposition states that a Fuchsian singular point is invariant under a limiting dilation process $Q' \mapsto \resc^{\{R_n\}}(Q)$, the second states that the regularity of a Fuchsian singular point with respect to a limiting dilated equation implies the regularity of the corresponding singular point for the original equation.
\begin{proposition}\label{lem:limitdil}
Assume that the equation
 \begin{equation*}
Q'(u):=-\pl (u) + V |u|^{p-2} u = 0 \qquad \mbox{ in }  X
\end{equation*}
 has a Fuchsian singularity at $\zeta \in \pd \hat X$, and let
  \begin{equation*}
\resc^{\{R_n\}}(Q)(u):=-\pl (u) + W |u|^{p-2} u = 0 \qquad \mbox{ in }  Y:=\lim_{n\to\infty} X/R_n,
\end{equation*}
 be a limiting dilated equation corresponding to
a sequence $R_n \to \zeta$. Then the equation
$\resc^{\{R_n\}}(Q)=0$ in $Y$ has Fuchsian singularities both at $0$ and
at $\infty$.
\end{proposition}
\begin{proof}
By definition, there exist $C>0$ and a relative neighborhood $X' \subset X$ of $\zeta$ such that \begin{equation}
 |x|^p |V(x)| \leq C \quad \mbox{for a. e. }  x\in X'.
\nonumber
\end{equation}
We claim that
\begin{equation}
 |x|^p |W(x)| \leq C \quad \mbox{for a. e. }  x\in Y.
\nonumber
\end{equation}
For $\varepsilon >0$, and $0<r<R<\infty$ consider the sets
$$A_{\varepsilon,r,R}^\pm:=\{ x\in Y\cap (B_R\sm B_r) \;\mid \;|x|^p W_\pm(x) \geq  C+\varepsilon \}.$$
We need to prove that for any $\varepsilon >0$ and $0<r<R<\infty$ we have  $|A_{\varepsilon,r,R}^\pm|=0$. Suppose that $|A_{\varepsilon,r,R}^+|>\delta$, and let $\chi$ be the indicator function of $A_{\varepsilon,r,R}^+$. Then for $n$ large enough we have
$$\int_Y \chi(x)|x|^p (W_+(x)-V_{R_n}(x))\, \mathrm{d} x\geq (C+\varepsilon)\delta - C\delta= \delta \varepsilon$$
which is a contradiction.
\end{proof}
\begin{proposition} \label{thm:dilatation}
Let $\zeta \in \hat{X}$ be a Fuchsian singular point of \eqref{eq:1}, and assume that
there is a sequence $R_n \to \zeta$, such that either $0$ or $\infty$
is a regular point of the limiting dilated equation $\resc^{\{R_n\}}(Q)=0$ in
$Y=\lim_{n\to\infty} X/R_n$. Then $\zeta$ is a regular
point of the equation $Q'(u)=0$ in $X$.
\end{proposition}


The proof of Proposition~\ref{thm:dilatation} appears in Section~\ref{sec:Harnack}. For a slightly stronger result see Proposition~\ref{prop:crit}. Next, we define a class of equations that dilates \eqref{eq:1} to the $p\,$-Laplace equation
after a finite number of iterations.

\begin{definition}\label{def:weak_Fuchs}{\em 
Let $V\in L^\infty_{\mathrm{loc}}(X)$ and $\zeta \in \pd \hat{X}$, where $\zeta=0$ or $\zeta=\infty$.  We say that $V$ has a {\bf weak Fuchsian singularity at $\zeta$}  if inequality \eqref{eq:FuchsCon} is satisfied, and in addition, there exist  $m$ sequences $\{R_n^{(i)}\}_{n=1}^\infty \subset \mathbb{R}_+$, $1\leq i\leq m$,  satisfying $R_n^{(i)} \to \zeta^{(i)}$, where $\zeta^{(1)} = \zeta$, and $\zeta^{(i)} = 0$ or
$\zeta^{(i)} = \infty$
 for $2\leq i \leq m$,  such that
\begin{equation}\label{eq:weakfuchs}
\resc^{R_n^{(m)}} \fundot \dots  \fundot \resc^{R_n^{(1)}}(Q)(w) = -\pl(w) \qquad \mbox{ on } Y,
\end{equation}
where $Y=\lim_{n\to\infty} X/R_n^{(1)}$.
 }
\end{definition}
Note that the potential $V$ considered in Theorem \ref{thm:weakI} (where $|x|^p V(x)$ is continuous near an  isolated point $\zeta$ and  $\lim_{|x| \to \gz} |x|^p V(x) = 0$) has a weak Fuchsian singularity at $\gz$. For further examples of potentials with weak Fuchsian singularities, see Remark~\ref{rem:VinLq}. On the other hand, the following example shows that a potential with a weak Fuchsian singularity might exhibits more complicated behaviors.
\begin{example}\label{ex:weakFuchs}{\em 
Let $R_n \to 0$ be a monotone sequence such that $R_{n+1}/R_{n} \to 0$. For
$n$ large enough put
\begin{equation} \nonumber
V(x) = \left\{
          \begin{array}{lll}
  |x|^{-p} &  \quad R_n \leq |x| < 2 R_n, \\[2mm]
    0      &  \quad  2 R_n \leq |x| < R_{n-1}.
          \end{array}\right.
 \end{equation}
Then $\resc^{R_n} \fundot \resc^{R_n}(Q)(w) = -\laplace_p(w)$ in $\mathbb{R}^d\sm \{0\}$,
however the potential corresponding to $\resc^{R_n}(Q)$ is nonzero in
$B_2\sm B_1$.
 }
\end{example}

The following theorem states that Conjecture~\ref{main_conj} holds true if $\zeta$ is
an isolated point of $\pd \hat{X}$, and  $V$ has a weak Fuchsian singularity at $\zeta$.
\begin{theorem}\label{thm:weak}
Suppose that (\ref{eq:1}) admits a positive solution, and let $\zeta \in \pd \hat{X}$ be an isolated point of $\pd \hat{X}$. Assume that $V$ has a weak Fuchsian singularity at $\zeta$.  Then
\begin{enumerate}
\item[(i)]
$\zeta$ is a regular point of (\ref{eq:1}).
\item[(ii)]
Equation (\ref{eq:1}) admits a unique positive solution
of  minimal growth in a neighborhood of $\bo \hat{X}\sm \{\zeta\}$.
\end{enumerate}
\end{theorem}
The proofs of Theorem~\ref{thm:weak} and of the other cases where Conjecture~\ref{main_conj} holds true are all along the following three main steps:
\begin{enumerate}
\item Rescale the original equation to obtain a limiting dilated equation $\resc^{R_n}(Q)(u)=0$ in $Y$.
\item Prove that either $0$ or $\infty$ is a regular point of the above limiting dilated equation.
\item Conclude that $\zeta$ is a regular point of the original equation and
hence the positive Liouville theorem holds.
\end{enumerate}
In the first step we use the Proposition \ref{lem:limitdil} and uniform Harnack inequalities (see Section \ref{sec:Harnack}). The third step is due to propositions \ref{thm:dilatation} and \ref{thm:RegUni}. The second step is the hardest. In general it requires
the strong comparison principle which is known to hold only in special
cases (see Section \ref{sec:comparison}). The following result is essential for proving step (2) for a weakly Fuchsian isolated singularity. Recall that for a weak Fuchsian isolated singular point the `final' limiting dilated equation is the $p\,$-Laplace equation in the punctured space.
 \begin{proposition}\label{thm:cor}
(i) If $p \leq d$, then $\zeta = 0$ is a regular point of the equation $-\pl(u) =0$ in $\mathbb{R}^d\sm \{0\}$.

(ii) If $p\geq d>1$, then $\zeta = \infty$ is a regular point of the equation $-\pl(u) =0$ in $\mathbb{R}^d$.
\end{proposition}
\begin{remark}\label{rem:reg} {\em 
It follows from Proposition~\ref{thm:dilatation} and Proposition~\ref{thm:cor} that for any $1<p<\infty$ and $d\geq 2$, both $\zeta = 0$ and $\zeta = \infty$ are regular points of the equation $-\pl(u) =0$ in $\mathbb{R}^d\sm \{0\}$ (cf. Remark~\ref{rem:reg1}).
 }
\end{remark}

The proof of Proposition~\ref{thm:cor} relies on the asymptotic behavior of  positive solutions near an isolated singularity. For $p \leq d$ the following (more general) result is known.
\begin{theorem}\label{thm:asymp}
Let $p \leq d$, and let $u$ be a positive solution of $Q'(u) =0$ in the punctured ball
$B_r\sm \{0\}$. Assume that $V\in L^\infty(B_r)$. Then either $u$ has a removable singularity at the origin, or
\begin{equation}\label{asymp:dgeqp}
u(x) \;\underset{x\to 0}{\sim}\;
\left\{
  \begin{array}{ll}
   \quad |x|^\frac{p-d}{p-1} & \hbox{ if }\; p<d,\\[2mm]
    -\log |x| & \hbox{ if }\; p=d.
  \end{array}
\right.
\end{equation}
\end{theorem}
We note that the proof of Theorem~\ref{thm:asymp} is a consequence of nontrivial asymptotic results by Serrin  \cite{Ser1} and V\'{e}ron (see \cite{ky3}; see also \cite{KV} for stronger results for the $p\,$-Laplace equation).

For the second part of Proposition~\ref{thm:cor} we need the following counterpart of Theorem~\ref{thm:asymp}. This result is of independent
interest and we shall discuss it in greater detail elsewhere.
\begin{theorem}\label{thm:asymp2}
Let $p\geq d>1$, and let $u$ be a positive solution of the equation $-\pl (u) =0$ in a neighborhood of
infinity in $\mathbb{R}^d$.  Then either $u$ has a  positive  limit as $x\to\infty$,
  or
\begin{equation*}
u(x) \;\underset{x\to \infty}{\sim}\;
\left\{
  \begin{array}{ll}
   \quad |x|^\frac{p-d}{p-1} & \hbox{ if }\; p>d,\\[2mm]
    \log |x| & \hbox{ if }\; p=d.
  \end{array}
\right.
\end{equation*}
\end{theorem}

The case $p=d$ in Theorem~\ref{thm:asymp2} follows from Theorem~\ref{thm:asymp} using the conformality of the $d$-Laplacian.
We prove the case $p>d$ of Theorem~\ref{thm:asymp2} in Appendix~\ref{sec:asymp}. The proof
uses a modified Kelvin transform and an argument similar to the one used in the proof of
Theorem~\ref{thm:asymp}.
\begin{remark}\label{rem:reg1} {\em 
Suppose that in a punctured neighborhood of $\gz$  \begin{align} \label{q}
   \qquad |V(x)| \leq \frac{M}{|x|^q} \,,
\end{align}
where  $M >0 $ and $q < p$ if $\zeta=0$ (respectively, $q > p$ if $\zeta=\infty$).

In \cite{MY3} we prove that if $V$ satisfies \eqref{q}, and if $u$ is a positive solution of the equation $-\pl (u) +V|u|^{p-2} u=0$ in a punctured neighborhood of $\gz=0$ and $p>d$ (respectively, $\gz=\infty$ and $p<d$),
then
$$\lim_{x \to \zeta} u(x) = C,$$
where $C$ is a nonnegative constant.
 }
\end{remark}

\vskip 3mm

\begin{proof}[Proof of Theorem~\ref{thm:weak}]
(i) By Proposition~\ref{thm:cor} either $0$ or $\infty$ is a regular point of the equation
$$
\resc^{R_n^{(m)}} \fundot \dots  \fundot \resc^{R_n^{(1)}}(Q)(w)= -\pl(w)=0\qquad \mbox{in }\; \mathbb{R}^d\sm \{0\}.
$$
Therefore, Proposition~\ref{thm:dilatation} and an induction argument imply that $\zeta$ is
a regular point of $Q$.

(ii) The claim follows from part (i) of the theorem and Proposition~\ref{thm:RegUni}.
\end{proof}

We conclude this section with general remarks concerning positive solutions of minimal growth in a neighborhood of infinity (see Definition \ref{def:minimal_gr}).

 \begin{remarks} \label{rem1}{\em 
(1) Suppose that the equation $Q'(u)=0$ admits a  positive solution in $X$. The existence of positive solutions of minimal growth in a neighborhood of infinity in $X$ follows by a simple exhaustion argument of solving Dirichlet problems in annular smooth domains $\{A_n\} \Subset (X\sm K)$ that exhaust $X\sm K$, subject to zero boundary condition on the (`exterior') portion of the boundary of $A_n$ that tends to infinity in $X$ (see, \cite{Agmon,Pin2}).

In particular, for any $x_0\in X$ the equation $Q'(u)=0$ admits a positive solution $u_{x_0}$ of the equation $Q'(u)=0$ in $X\sm \{x_0\}$ of minimal growth in a neighborhood of infinity in $X$ \cite{Agmon,Pin2}. This solution is known to be unique
 if $1<p \leq d$ \cite{Pin2}. The uniqueness for $p>d$ follows from Theorem~\ref{thm:weak} (see Corollary~\ref{cor:dlessp}). The equation $Q'(u)=0$ is {\bf critical} (respectively, {\bf subcritical}) in $X$ if such a solution has a removable (respectively, nonremovable) singularity at $x_0$ \cite{Pin2}.

On the other hand, for any $\zeta\in \pd \hat{X}$ the equation $Q'(u)=0$ admits a positive solution in $X$ of minimal growth in a neighborhood of $\bo \hat{X}\sm \{\zeta\}$; the proof is similar to the proof in Section~5 of \cite{Pin1} using a Martin sequence of the form $\{u_{x_n}(x)/u_{x_n}(x_0)\}$,  where $x_0$ is some fixed reference point in $X$, $\{x_n\}$ is some sequence in $X$ such that $x_n\to\zeta$, and $u_{x_n}$ is the positive solution of the equation $Q'(u)=0$ in $X\sm \{x_n\}$ of minimal growth in a neighborhood of infinity in $X$.

(2) A (global) positive solution of the equation $Q'(u)=0$ in $X$ of minimal growth in a neighborhood of infinity in $X$ is called a {\bf ground state of the equation $Q'(u)=0$ in $X$}. Note that a ground state is a positive solution of minimal growth in a neighborhood of  $\bo \hat{X'}\sm \{x_0\}$,  where $x_0$ is some point in $X$, and $X':=X\sm \{x_0\}$. It follows that \eqref{eq:1} admits a ground state if and only if \eqref{eq:1} is critical in $X$. Moreover, \eqref{eq:1} is critical in $X$ if and only if it admits a unique global positive supersolution \cite{Pin2}. In particular, the positive Liouville theorem holds true in the critical case.

(3) Let $\Gamma$ be a $C^2$-portion of $\bo X\sm \{\zeta\}$, and let $U\subset X$ be a relative neighborhood of $\Gamma$. Assume that $V\in L^\infty_{\mathrm{loc}}(X)\cap L^\infty(X\cap U)$. Then positive solutions of minimal growth vanish continuously on $\Gamma$. Moreover, if $X\subset\mathbb{R}^d$ is an unbounded $C^2$-domain and $V\in L^\infty_{\mathrm{loc}}(\bar{X})$, then any positive solution of \eqref{eq:1} which vanishes continuously on $\bo X$ is a positive solution of  minimal growth in a neighborhood of $\bo \hat{X}\sm \{\infty\}$.
 }
   \end{remarks}

\section{Weak and strong comparison principles and applications}\label{sec:comparison}
In this section we discuss the validity of the weak comparison principle (WCP) and the strong comparison principle (SCP) for the equation $Q'(u)=0$, and their relations to the results of the present paper. Roughly speaking, the validity of the SCP
 implies that Conjecture~\ref{main_conj} holds true (see Theorem~\ref{thm:assume_scp}).
\begin{theorem}[Weak comparison principle \cite{Garcia}]\label{thm:wcp}
Let $\Omega'$ be a bounded $C^{1,\alpha}$ subdomain of a domain $\Omega\subset X$, such that $\Omega'\Subset \Omega$.
Assume that the equation $Q'(w)=0$ admits a positive solution in $\Omega$ and
suppose that $u,v\in C^1(\Omega')\cap C(\overline{\Omega'})$, $u,v\geq 0$ satisfy the following inequalities
\begin{eqnarray}\label{eq:wcp}
      Q'(v) \geq 0  & \mbox{in} & \Omega', \nonumber\\
      Q'(u) \leq 0 & \mbox{in} & \Omega', \\
 u \leq v & \mbox{on} & \bo \Omega' \nonumber.
\end{eqnarray}
Then $u \leq v$ in $\Omega'$.
\end{theorem}
Next, we state a conjecture concerning the strong comparison principle (SCP), and discuss some  cases where it holds.
\begin{conjecture}[Strong comparison principle]\label{con:SCP}
Suppose that all the conditions of Theorem~\ref{thm:wcp} are satisfied. Then
$$
\mbox{either} \quad u<v \quad \mbox{in} \quad \Omega', \quad \mbox{or}\quad
u = v \quad \mbox{in} \quad \Omega'.
$$
\end{conjecture}
\begin{remarks} \label{rem:3}{\em 
(1) By the strong maximum principle, the SCP holds true for a general (nonsymmetric) {\em linear} operators, and in particular, Conjecture~\ref{con:SCP} holds for $p=2$. By the same reason, it also holds for $1<p<\infty$ and $u=0$  \cite{Garcia}.  For other particular cases where the SCP holds true see \cite{Barles,Tak,Damascelli,Sciunzi,deF,Tolks} and the references therein. On the other hand, the validity of Conjecture~\ref{con:SCP} for $p\ne 2$ is an open problem even for the simplest case where $u$ and $v$ are positive $p\,$-harmonic functions, i.e. positive {\em solutions} of the
equation $-\pl (w)=0$ in $\mathbb{R}^d$, and $d>2$ (see e.g. \cite{pNotes}). Conjecture~\ref{con:SCP} for $p\ne 2$ and strictly positive potentials  $V$ is open as well.

(2) By \cite{Garcia}, WCP holds true under slightly weaker assumptions. Namely, WCP holds true if condition (\ref{eq:wcp}) is replaced with
\begin{eqnarray}\label{eq:wcps}
          Q'(v) \geq 0 & \mbox{in} & \Omega', \nonumber\\
      Q'(v) \geq Q'(u) & \mbox{in} & \Omega',  \\
      u \leq v & \mbox{on} & \bo \Omega' \nonumber .
\end{eqnarray}
However, there is a counterexample \cite{Tak} which shows that the corresponding SCP does not hold under condition \eqref{eq:wcps}.
The construction of the counterexample in \cite{Tak} relies on the fact that \eqref{eq:wcps} is not homogeneous under scaling. More precisely, there might exist a constant $C>1$ such that \eqref{eq:wcps} holds true, $Cu\leq v$ on $\bo \Omega'$, but $Q'(v) \geq CQ'(u)$ does not hold.
 }
\end{remarks}
Next, we state a special case where SCP is known to hold (cf. \cite{Damascelli}). To this end, we need to discuss the set $S_u$ of critical points of a given function $u$. More precisely, let $u\in C^1(\Omega')$, then
$$
S_u:=\{x \in \Omega'\mid |\grad u(x)| = 0\}
$$
is called the {\bf set of critical points of $u$}.
For $u, v\in C^1(\Omega')$, we use the notation $S:=S_u \cap S_v$,  and
by $S^c$ we denote the complement of $S$ in $\Omega'$, the {\bf nondegeneracy set}. Note that $S$ is a  closed set \cite{Dib}.
\begin{theorem}\label{thm:SCP}
Suppose that $\Omega,\,\Omega',\,u,\,v$ satisfy all the conditions of Theorem~\ref{thm:wcp}.
Assume further that
\begin{enumerate}
\item[(i)] $ S^c $ is a connected set.
\item[(ii)] $S \cap \bo \Omega'$ is empty.
\end{enumerate}
Then the SCP holds true  with respect to $u$ and $v$.
\end{theorem}
\begin{proof}
Let $E\subset \Omega'$ be a set where $u(x) = v(x)$. If $E \cap S^c$ is empty, then by (ii)
there exists a $0<C<1$ such that $u\leq Cv<v$  on the boundary of $\Omega'$
and by the WCP $u\leq Cv<v$ in $\Omega'$.

So, let us assume that $E \cap S^c$ is not empty and choose $x \in E \cap S^c$.
Without loss of generality we may assume that $|\grad v(x)| > 0$,  and let
$B_r(x)\subset \Omega'$ be  a ball such that $B_r(x)\cap S_v=\emptyset$. Then
$$
A_{ij} := |\grad v|^{p-4} \left((p-2) \pd_i v \pd_j v + |\grad v|^2 \delta_{ij}\right)
$$
is a positive definite matrix in $B_r(x)$. Indeed, this is trivial for $p \geq 2$ and
follows from the inequality
$$
\sum_{i,\,j} (p-2) \pd_i v \pd_j v \xi_i \xi_j \geq - |\grad v|^2 |\xi|^2 \qquad
\forall \xi \in \mathbb{R}^d,
$$
for $1<p<2$. With this condition
the tangency principle (see Theorem 2.5.2 in \cite{Ser})
implies that $B_r(x) \subset E$ (the proof relies on the strong maximum principle for the function $u-v$ that satisfies a certain elliptic linear equation near the nondegenerate point $x$).
Thus, $E$ is a nonempty open set in $S^c$. Since at the same time it is  a closed set in $S^c$, the assumption that $S^c$ is connected implies $\overline{S^c} \subset E$.
On the other hand, $X \sm \overline{S^c}$ is an open set where $\grad u = \grad v =0$.
Consequently, the claim follows since $u$ and $v$ are constant on any connected component
of this set.
\end{proof}
In Section~\ref{sec:proofs} we use Theorem~\ref{thm:SCP} to prove the first part of Theorem~\ref{thm:sym} concerning regular points of spherically symmetric equations (cf. \cite{Sciunzi}).

The main obstacle in the proof of Conjecture~\ref{main_conj} in the quasilinear case is that the SCP
 is not known to hold for such equations without some artificial
assumptions on the critical set. In fact, our result concerning the spherically symmetric case (Theorem~\ref{thm:sym}) can be viewed as a variation of the following general result.
\begin{theorem}\label{thm:assume_scp}
Assume that \eqref{eq:1} admits a positive solution, and suppose that $V$ has a Fuchsian singularity at an isolated point  $\zeta\in \pd \hat{X}$. Assume further that
there is a sequence $R_n \to \zeta$ such that SCP holds true with respect to any two positive global solutions of the
limiting dilated equation
\begin{equation}\label{eq:lde}
\resc^{\{R_n\}}(Q)(w) = 0 \qquad \mbox{in }\, Y.
\end{equation}
Then
\begin{enumerate}
\item[(i)] $\zeta$ is a regular point of the Eq.~\eqref{eq:1}.
\item[(ii)] Equation \eqref{eq:1} admits a unique positive solution
of minimal growth in a neighborhood of $\bo \hat{X} \sm \{\zeta\}$.
\end{enumerate}
\end{theorem}
We prove Theorem~\ref{thm:assume_scp} in Section~\ref {sec:proofs}.
The extension of Theorem~\ref{thm:assume_scp} to the case where $\zeta$ is a nonisolated singular point is studied in Theorem~\ref{thm:assume_scpB}.

\section{Uniform Harnack inequality and behavior near regular points}\label{sec:Harnack}
One of our main tools in the present paper is dilatation. Let $u$ be a positive solution of (\ref{eq:1}). Then for any $R>0$ the function  $u_R(x) := u(Rx)$
is a solution of the equation
\begin{equation}
Q'_R(u_R):= -\pl (u_R) + V_R(x) |u_R|^{p-2}u_R = 0 \qquad \mbox{in} \quad X/R,
\end{equation}
where $V_R(x) := R^p V(Rx)$ is the scaled potential.

Let $u,\,v$ be two positive solutions of the equation $Q'(u)=0$ in some relative neighborhood $X'\subset X\subset \hat{X}$ of $\zeta$, where $\zeta= 0$ (respectively, $\zeta = \infty$). Assume that $u$ and $v$ vanish continuously on $(\bo X'\cap \bo X) \sm \{\zeta\}$. So, $u,v\in \mathcal{G}_\gz$.
Let $A_r$ be the annular set $A_r:=(B_r \sm \overline{B_{r/2}})\cap X'$, and denote $$
\mathbf{a}_r:= \inf_{x\in A_r}\frac{u(x)}{v(x)}, \qquad  \mathbf{A}_r:=\sup_{x\in A_r} \frac{u(x)}{v(x)}.
$$
Then by the local Harnack inequality \cite{Ser} and the boundary Harnack inequality \cite{Ver,Lewis} there exists $r_0>0$ such that $0<\mathbf{a}_r \leq \mathbf{A}_r <\infty$ for all $0<r<r_0$ (respectively, $r>r_0$). For a Fuchsian type singularity we have
\begin{lemma}[Uniform Harnack inequality]\label{lem:BHI}
 Assume that the potential $V$ has a Fuchsian type singularity at $\zeta$, where $\zeta\in \{0,\infty\}$.
Then there exists $C>0$ such that  the following uniform boundary Harnack inequality holds
$$
\mathbf{A}_r \leq C  \mathbf{a}_r
$$
for any $u,v\in \mathcal{G}_\gz$ and $r$ near $\gz$.
\end{lemma}
\begin{proof}
For $r>0$, denote
$$u_r(x):=u(rx),\quad v_r(x):=v(rx) \qquad x\in X'/r.$$
Also, for $r>0$ consider the annular set $\tilde{A}_r: = (B_{2r} \sm \overline{B_{r/4}})\cap X'$.
Note that if $\zeta=0$ (respectively, $\zeta=\infty$) is an isolated singular point, then for $r<r_0$ (respectively, $r>r_0$) $A_r/r$ and $\tilde{A}_r/r$ are fixed annuli $A$ and  $\tilde{A}$ satisfying $A \Subset \tilde{A}$. Similarly, if $\zeta = \infty$ is not an isolated singular point, then without loss of generality, $X'$ is a truncated cone. Hence for $r$ large enough, the domains $A_r/r$ and $\tilde{A}_r/r$ are fixed $C^2$ domains $A$ and $\tilde{A}$, respectively,  satisfying $A \subset \tilde{A}$.

It follows that for such $r$ the functions $u_r$ and $v_r$ are positive solutions of the equations $Q'_r(w)=0$ in $\tilde{A}$ (which, in the case where $\zeta = \infty$ is not an isolated singular point, vanish continuously on $\pd X'/r\cap \pd \tilde{A}$). Since $V$ has a Fuchsian type singularity at $\zeta$, it follows that the scaled potentials $V_r$ are uniformly bounded in $\tilde{A}$. Therefore,
either the local Harnack inequality \cite{Ser} or the boundary Harnack inequality  \cite{Ver,Lewis} in $\tilde{A}$ implies that
$$
\mathbf{A}_r =\sup_{x\in A_r} \frac{u(x)}{v(x)} =\sup_{x\in A} \frac{u_r(x)}{v_r(x)} \leq C \inf_{x\in A} \frac{u_r(x)}{v_r(x)} = C \inf_{x\in A_r} \frac{u(x)}{v(x)} = C  \mathbf{a}_r,
$$
where $C$ is $r$ independent.

If $\zeta = 0$  belongs to a $C^2$-portion of $\pd X$, then the result follows directly from the boundary Harnack inequalities of \cite{Ver,Lewis}.
  \end{proof}
The following lemma gives additional information on the behavior of the quotient of two positive solutions near the singular point $\zeta$.
\begin{lemma}\label{lem:mon}
Let $u,v\in \mathcal{G}_\gz$ defined in  a relative neighborhood $X'$   of $\zeta$, where $\zeta\in \{0, \infty\}$). For $r>0$ denote
$$
m_r:=\inf_{S_r \cap X'} \frac{u(x)}{v(x)}\,,\qquad M_r:=\sup_{S_r \cap X'} \frac{u(x)}{v(x)}\,.
$$

(i) The functions $m_r$ and $M_r$ are monotone as $r\to \zeta$, In particular, there are numbers $0\leq m\leq M \leq \infty$ such that
\begin{equation}\label{eq:mM}
\lim_{r\to \zeta} m_r = m, \qquad \lim_{r\to \zeta} M_r = M.
\end{equation}

(ii) Suppose further that $u$ and $v$ are positive solutions of the equation $Q'(u)=0$ in $X$
of minimal growth in $\bo X\sm \{\zeta\}$, then
$0< m\leq M < \infty$ and $m_r\searrow  m, \,M_r\nearrow  M$ as $r\to \zeta$.
\end{lemma}
\begin{proof} Assume first that $\zeta = 0$.

(i) Suppose that $B_{R_0} \cap X\subset X'$. Let $\{r_n\}_{n=0}^\infty$ be a strictly decreasing sequence such that $r_n\leq R_0$, and $\lim_{n\to \infty} r_n=0$.
Denote also $m_n=m_{r_n}$ and $M_n=M_{r_n}$.
By Harnack, $m_{n}$ and $M_n$ are positive.

By the weak comparison principle, for any $n\geq 1$ and $1\leq i\leq j$ we have
\begin{equation}\label{eq:conv}
 \min(m_{n},m_{n+j})\leq m_{n+i}, \qquad \max(M_{n},M_{n+j})\geq M_{n+i}.
\end{equation}
So, $\{m_n\}$ and $\{M_n\}$ are ``concave" and ``convex" sequences, respectively. It follows that the sequences $\{m_n\}$ and $\{M_n\}$ are finally monotone.
Thus, $m_r$, and $M_r$ are finally monotone functions of $r$. In particular, $\lim_{r\to 0} m_r$, and $\lim_{r\to 0} M_r$ exist.

\vskip 3mm

(ii) By the definition of positive solutions of minimal growth, for any
$r < r'$, the inequality $m_r v(x) \leq u(x)$ on $S_r \cap X$ implies
$ m_r v(x) \leq u(x)$ on $S_{r'}\cap X$,  and therefore, $m_{r} \leq m_{r'}$. By a similar argument
$M_{r'} \leq M_{r}$. Consequently, $m_r\searrow  m$, and $M_r\nearrow  M$ as $r\to 0_+$. By Lemma~\ref{lem:BHI}, $m_r \asymp M_r$, and therefore $0<m\leq M < \infty$.

The case $\zeta = \infty$ follows by the same argument.
\end{proof}
The second part of Lemma~\ref{lem:mon} readily implies the following corollary:
\begin{corollary}\label{cor1}
Let $\zeta\in \pd \hat{X}$. Assume that Eq.~\eqref{eq:1} admits a positive solution and has a Fuchsian type singularity at $\zeta$. Let $u,\,v$ be two positive solutions  of \eqref{eq:1} of minimal growth in a neighborhood of $\bo \hat{X} \sm \{\zeta\}$. Then $u \asymp v$. More precisely,
$$ m v(x) \leq u(x) \leq M v(x) \qquad x \in X, $$
where $0<m\leq M<\infty$ are given by \eqref{eq:mM}.
\end{corollary}
\begin{remark}{\em 
Lemmas~\ref{lem:BHI} and \ref{lem:mon} imply that for $u,v \in G_\gz$ either
$$
 u \succ v \quad \mbox{or} \quad u \prec v \quad \mbox{or} \quad u \asymp v.
$$
 }
\end{remark}

We are now ready to prove Proposition~\ref{thm:RegUni} that claims that the regularity of $\zeta$ implies the uniqueness statement of Conjecture~\ref{main_conj}.
\begin{proof}[Proof of Proposition~\ref{thm:RegUni}]
 Let $u$ and $v$ be two positive solutions of (\ref{eq:1})
of minimal growth in a neighborhood of $\bo \hat{X}\sm \{\zeta\}$.
By Corollary~\ref{cor1}
$$
m v(x) \leq  u(x) \leq M v(x) \qquad x\in X,
$$
where $m$ and $M$ are positive numbers given by \eqref{eq:mM}.
By our assumption $\zeta$ is a regular point. Hence,
\begin{equation*}
    \underset{x\in X}{\underset{x \to \zeta }{\lim}} \;\;\frac{u(x)}{v(x)} \quad \mbox{exists}.
    \end{equation*}
Therefore, $m=M$, and $u(x) = m v(x)$.
\end{proof}
Next we prove Proposition~\ref{thm:dilatation} concerning a regular point of a limiting dilated equation.
\begin{proof}[Proof of Proposition~\ref{thm:dilatation}]
 Suppose that $V$ has a Fuchsian type singularity at the point $\zeta\in \bo \hat{X}$, and let $u,v\in \mathcal{G}_\gz$.
 %
 %
Let
\begin{equation}\label{eq:mrMr}
m_r:=\inf_{S_r \cap X'} \frac{u(x)}{v(x)}\,,\qquad
M_r:=\sup_{S_r \cap X'} \frac{u(x)}{v(x)}\,. \end{equation}
If $ M :=\lim_{r \to \zeta} M_r =  \infty$ (respectively, $ m :=\lim_{r \to \zeta} m_r = 0$), then by Lemma \ref{lem:BHI} $m=\infty$ (respectively, $M=0$) and the limit
\begin{equation}\nonumber
\underset{x\in X'}{\underset{x \to \zeta }{\lim}} \;\; \frac{u(x)}{v(x)} \qquad \mbox{exists in the generalized sense}.
\end{equation}
So, we may assume that  $u\asymp v$ in some  neighborhood $X'\subset X$ of $\zeta$.
Let $R_n\to\zeta$ be a sequence associated with the limiting dilated equation
\begin{equation}\label{eq:limitdil}
-\pl (w) + W|w|^{p-2}w = 0 \quad \mbox{in} \qquad Y,
\end{equation}
where $W$ is the weak$^*$ limit in $L_{\mathrm{loc}}^{\infty}(Y)$ of  $V_n:=V_{R_n}$,  the associated scaled potential (see (\ref{eq:wfuchs})).  Fix $x_0\in \mathbb{R}^d$ such that
$R_n x_0\in X$ for all $n\geq 1$. Define
\begin{equation}\label{eq:unvn}
u_n(x):=\frac{u(R_n x)}{u(R_n x_0)},\qquad v_n(x):=\frac{v(R_n x)}{u(R_n x_0)}.
\end{equation}
Then $u_n$ and $v_n$ are positive solutions of the equation
$$
-\pl (w) + V_n(x) |w|^{p-2}w = 0 \quad \mbox{in} \quad X'/R_n.
$$
Since $u_n(x_0)=1$ and $v_n(x_0) \asymp 1$, it follows by a standard elliptic argument that  $\{R_n\}$ admits a subsequence (denoted again by $\{R_n\}$) such that
\begin{equation}\label{eq:local_unif}
\lim_{n \to \infty}u_n(x)=u_\infty(x), \quad  \mbox{ and } \quad  \lim_{n \to \infty}v_n(x)=v_\infty(x)
\end{equation}
locally uniformly in $Y$, and  $u_\infty$ and $v_\infty$ are positive solutions of
the limiting dilated equation \eqref{eq:limitdil} that vanish continuously on $\bo Y \sm \{0\}$ (this follows from  \cite{Ver} and \cite{Lieberman}).
Moreover, for any fixed $R>0$ we have
$$
\sup_{x\in S_R} \frac{u_\infty(x)}{v_\infty(x)}
=  \sup_{x\in S_R} \lim_{n \to \infty}\frac{u_n(x)}{v_n(x)} =
\lim_{n \to \infty} \sup_{x\in S_R} \frac{u_n(x)}{v_n(x)} =
\lim_{n \to \infty} M_{RR_n} = M,
$$
where we used the local uniform convergence of the sequence $\{u_n/v_n\}$
in  $Y$, and the existence of $\lim_{r \to \zeta}M_r=M$.
Similarly,
$$
\inf_{x\in S_R} \frac{u_\infty(x)}{v_\infty(x)} = m.
$$
By our assumption either $\zeta_1 = 0$ or $\zeta_1 = \infty$ is a regular point of
the limiting dilated equation \eqref{eq:limitdil}, and so
$$
 \underset{x\in Y}{\underset{x \to \zeta_1 }{\lim}} \;\;\frac{u_\infty(x)}{v_\infty(x)} \quad \mbox{ exists}.
$$
Therefore, $m=M$, which in turn implies that
\begin{equation*}
    \underset{x\in X'}{\underset{x \to \zeta }{\lim}} \;\;\frac{u(x)}{v(x)} \quad \mbox{ exists}
    \end{equation*}
and so $u$ and $v$ are comparable with respect to $\lsim$ near $\zeta$.
\end{proof}
The following is a slightly stronger version of Proposition~\ref{thm:dilatation}.
\begin{proposition}\label{prop:crit}
Suppose that $\zeta \in \pd \hat{X}$ is a Fuchsian singular point with respect to the equation $Q'(u)=0$ in $X$, and assume that
there is a sequence $R_n \to \zeta$, such that for any two {global} positive solutions $u_\infty, v_\infty$ of
the limiting dilated equation $\resc^{\{R_n\}}(Q)=0$ in
$Y$ that vanish on $\pd Y\sm \{0\}$ and for either $\gz_1=0$ or $\gz_1=\infty$
$$
 \underset{x\in Y}{\underset{x \to \zeta_1 }{\lim}} \;\;\frac{u_\infty(x)}{v_\infty(x)} \quad \mbox{ exists}.
$$
 Then the assertions of Conjecture~\ref{main_conj} hold true for  \eqref{eq:1} and $\gz$.
In particular, if a limiting dilated equation of \eqref{eq:1} is critical in $Y$, then the assertions of Conjecture~\ref{main_conj} hold true for  \eqref{eq:1} and $\gz$.
\end{proposition}
\begin{proof}
The proof follows from the simple observation that in the proof of Proposition~\ref{thm:dilatation},  the limit
$$
 \underset{x\in Y}{\underset{x \to \zeta_1 }{\lim}} \;\;\frac{u_\infty(x)}{v_\infty(x)}
$$
should exist only for any two global positive solutions of
the limiting dilated equation  $\resc^{\{R_n\}}(Q)=0$ in
$Y$ that vanish on $\pd Y\sm \{0\}$.
\end{proof}

\begin{remark}\label{remark}{\em 
Fix a Fuchsian singular point $\zeta \in \pd \hat X$, and $u\in \mathcal{G}_\gz$. Assume that
there is a sequence $R_n \to \zeta$, such that either $0$ or $\infty$
is a regular point with respect to the solution $u_\infty$ of the limiting dilated equation $\resc^{\{R_n\}}(Q)=0$ in $Y:=\lim_{n\to\infty} X/R_n$, where $u_\infty$ is the limit of the sequence $\{u_n\}$ defined by \eqref{eq:unvn}. Then $\zeta$ is a regular
point with respect to the solution $u$ of the equation $Q'(u)=0$ in $X$. The proof of this statement is the same as the proof of Proposition~\ref{thm:dilatation}.
 }
\end{remark}
\section{Proofs of theorems~\ref{thm:sym} and \ref{thm:assume_scp}}
\label{sec:proofs}
We start with the proof of Theorem~\ref{thm:sym} concerning Liouville's theorem in the
spherically symmetric case.  First we need the following lemma.
\begin{lemma}\label{lem:sym}
 Suppose that the domain $X$ and the potential $V$ are spherical
symmetric, and that Eq.~\eqref{eq:1} admits a positive
solution. Assume further that $V$ has a Fuchsian type singularity at
$\zeta$, where $\zeta=0$ or $\zeta=\infty$.  Then
for any $u\in \mathcal{G} _\gz$ there exists a radial solution $\tilde{u}\in \mathcal{G} _\gz$ such that $u\approx \tilde{u}$.
\end{lemma}
 \begin{proof}
 Assume that $\gz=0$, and let $u\in \mathcal{G} _\gz$. Let $R$ be a fixed positive number such that $u$ is defined in $B_{2R}\sm \{0\}$. For $0<r\leq R$ denote $m_r:=\inf_{S_r} u(x)$, and consider the solution $u_r$ of the following Dirichlet problem
 \begin{equation}\label{eq:rad5}
\left\{
  \begin{array}{ll}
Q'(w) = 0 \qquad \mbox{in } B_R\sm \overline{B_r} ,\\[1mm]
\;\;w(x)= m_R \qquad |x|=R,     \\[1mm]
\;\;w(x)= m_r \qquad |x|=r.
  \end{array}
\right.
 \end{equation}
Clearly $u_r$ is spherical symmetric.   Moreover, by the uniform Harnack inequality and the WCP there exists a constant $C$ independent of $r$ such that $Cu\leq u_r\leq u$ in $B_R\sm B_r$. Consequently, there exists a sequence of $r_n\to 0$ such that $u_{r_n}\to \tilde{u}$ locally uniformly in $B_R\sm \{0\}$. Clearly, $\tilde{u}$ is a radial positive solution of the equation $Q'(w)=0$ in $B_R\sm \{0\}$ satisfying $\tilde{u}\approx u$ near the origin.

 The proof for $\gz=\infty$ is similar and left for the reader.
 \end{proof}
\begin{proof}[Proof of Theorem~\ref{thm:sym}]
(i) and (iii)  Let $u$ and $v$ be two positive solutions of the equation $Q'(u)=0$  in a neighborhood $X' \subset X$ of $\zeta$, and assume first that  $u$ is spherically symmetric. We follow along the lines of the proof of Proposition~\ref{thm:dilatation}.

Lemmas \ref{lem:BHI} and \ref{lem:mon} imply that either
\begin{equation}\nonumber
\underset{x\in X'}{\underset{x \to \zeta }{\lim}} \;\; \frac{u(x)}{v(x)}\in \{0,\infty\},
\end{equation}
and the limit exists in the generalized sense, or $u\asymp v$ in a neighborhood $X'' \subset X$ of $\zeta$.

Therefore, we may assume that $u$ and $v$ are comparable. Let $\{R_n\}$ be a sequence such that $R_n\to \zeta$ and
\begin{equation*}
\lim_{n \to \infty}u_n(x)=u_\infty(x), \quad  \mbox{ and } \quad  \lim_{n \to \infty}v_n(x)=v_\infty(x),
\end{equation*}
where $u_n$ and $v_n$ are the corresponding dilated normalized solutions defined as in \eqref{eq:unvn}, and $u_\infty,\,v_\infty$ are solutions of a limiting dilated equation
\begin{equation*}
\resc^{\{R_n\}}(Q)(w) = -\pl(w) + W(|x|)|w|^{p-2} w = 0\qquad  x\in \mathbb{R}^d\sm \{0\}.
\end{equation*}
Clearly,  $W$ and $u_\infty$ are spherically symmetric. As in the proof of Proposition~\ref{thm:dilatation}, for any fixed $R>0$ we have
$$
\inf_{x\in S_R} \frac{u_\infty(x)}{v_\infty(x)} = m,
\qquad \sup_{x\in S_R} \frac{u_\infty(x)}{v_\infty(x)}
= M,
$$
where as usual, $m= \lim_{r\to \zeta} m_r$, and $M= \lim_{r\to \zeta} M_r$, and $m_r,\;M_r$ are defined by \eqref{eq:mrMr}. Moreover, we may assume that $W$ is nontrivial near $\zeta$, otherwise, $W$ has a weak Fuchsian singularity at $\zeta$, and the result follows from part (i) of Theorem~\ref{thm:weak}.

The set $S_{u_\infty}$ of the critical points of $u_\infty$ is closed and spherically symmetric. Therefore, one of the following two cases occur:
\begin{enumerate}
\item $\zeta$ is an interior point of $\widehat{S_{u_\infty}}$. Then there is a neighborhood $\Omega_\zeta$ of $\zeta$ such that $|\grad u_\infty| = 0$ in $\Omega_\zeta$.
So, $u_\infty$ is constant near $\zeta$, but this contradicts the non-triviality of $W$ near $\zeta$.
\item There exists an  annulus $A=B_{R}\sm B_{r}$ sufficiently close to $\zeta$ such that $S_{u_\infty}\cap A=\emptyset$. Therefore,  Theorem~\ref{thm:SCP} implies that the SCP holds true in $A$. Thus,
$m v_\infty = u_\infty = M v_\infty$ in $A$. So, $m=M$, and the proposition follows.
\end{enumerate}
Now let $u,v\in \mathcal{G}_\gz$. By Lemma~\ref{lem:sym} there exists a radial solution $\tilde{u}\in \mathcal{G}_\gz$ such that $\tilde{u}\approx u$. By the proof above, we have that $\tilde{u} \sim u$ and
\begin{equation}\nonumber
\underset{x\in X'}{\underset{x \to \zeta }{\lim}} \;\; \frac{\tilde{u}(x)}{v(x)}
\end{equation}
 exists in the generalized sense. Therefore,
\begin{equation}\nonumber
\underset{x\in X'}{\underset{x \to \zeta }{\lim}} \;\; \frac{u(x)}{v(x)}
\end{equation}
 exists in the generalized sense.

(ii) We claim that there exists  a spherically symmetric positive solution of the equation $Q'(w) = 0$ in $X$ of minimal growth in a neighborhood of $\bo \hat{X}\sm \{\zeta\}$.

Indeed, let  $\mathcal{B}_n\Subset \mathcal{B}_{n+1}\Subset X$, $n=1,2,\ldots$,  be  a sequence of radially symmetric domains of $X$ that exhausts $X$, and let $\pd_\pm \mathcal{B}_n$ be the connected components of $\pd \mathcal{B}_n$, such that $\pd_- \mathcal{B}_n \to \zeta$
 (if $X=\mathbb{R}^d$, then $\pd \mathcal{B}_n$ has only one connected component, in this case $\pd_+ \mathcal{B}_n=\emptyset$). Fix a point $x_0 \in B_0$ and
consider the sequence $\{u_n\}$ of the (spherically symmetric) solutions of the following Dirichlet problems
\begin{equation*}
\left\{
  \begin{array}{ll}
Q'(w) = 0 \qquad \mbox{in } \mathcal{B}_n ,\\[1mm]
\;\;w(x)= 0 \qquad \mbox{on } \pd_+ \mathcal{B}_n,     \\[1mm]
\;\;w(x)= C_n \qquad \mbox{on } \pd_- \mathcal{B}_n,
  \end{array}
\right.
 \end{equation*}
where $C_n$ is a positive constant such that $u_n(x_0) =1$.
It follows that $\{u_n\}$ admits a subsequence that converges locally uniformly to a nonnegative solution $u$ of the equation $Q'(w) = 0$ in $X$. By construction (cf. Remarks~\ref{rem1}),  $u$  is a positive radial solution of minimal growth in a neighborhood of $\bo \hat{X}\sm \{\zeta\}$.

Corollary~\ref{cor1} implies that any other positive solution of minimal growth in a
neighborhood of $\bo \hat{X}\sm \{\zeta\}$ is comparable to the above radial solution $u$.  Therefore the uniqueness follows readily from part (i).
\end{proof}
\begin{remark}\label{rem:8} {\em 
Proposition~\ref{thm:dilatation} and Theorem~\ref{thm:sym} imply the following useful result.

Suppose that the potential $V$ is of the form $V=V_1+V_2$, where $V_1$ is spherical
symmetric, and has a Fuchsian isolated singularity at
$\zeta$, where $\zeta=0$ or $\zeta=\infty$, and $V_2$ has a weak Fuchsian singularity at
$\zeta$. Assume further that Eq.~\eqref{eq:1} admits a positive
solution.   Then

(i) $\zeta$ is a regular point of \eqref{eq:1}.

(ii) Equation \eqref{eq:1} admits a unique positive solution of minimal growth in a
neighborhood of infinity in $X \sm \{\zeta\}$.
 }
\end{remark}

Finally, we prove Theorem \ref{thm:assume_scp}, which claims that Conjecture~\ref{main_conj} holds true under the assumption that the SCP holds true for a limiting dilated equation.
\begin{proof}[Proof of Theorem~\ref{thm:assume_scp}]
We prove only the first part of the theorem. The second part
follows by Proposition~\ref{thm:RegUni}. Not surprisingly,
the proof of the first part is similar to the proof of
Proposition~\ref{thm:dilatation}.

Recall the definition of  $m_r$ and $M_r$  (see \eqref{eq:mrMr}). We need to prove that $m=M$, where $M :=\lim_{r \to \zeta} M_r$, and $m :=\lim_{r \to \zeta}m_r$.

If  $M =  \infty$ (respectively,  $m = 0$), then by Lemma~\ref{lem:BHI}, $m=\infty$ (respectively, $M=0$) and the statement follows.

So, let us assume that $0<m\leq M<\infty$. Accordingly $u \asymp v$ in $X'$,
and therefore (after choosing a subsequence of the given sequence $\{R_n\}$) the normalized dilated sequences
$\{u_n\}$ and $\{v_n\}$ defined by \eqref{eq:unvn} converge locally uniformly to positive solutions $u_\infty$ and $v_\infty$ of the dilated equation
\begin{equation}\label{eq:eq_with_scp1}
-\pl (u) + W(x) |u|^{p-2}u = 0 \qquad \mbox{in }  Y.
\end{equation}
Using the same argument as in the proof of Proposition \ref{thm:dilatation}, we have for any $R>0$
 $$
 m=\inf_{x \in S_R \cap Y} \frac{u_\infty(x)}{v_\infty(x)}\,,
 \qquad
 M=\sup_{x \in S_R \cap Y} \frac{u_\infty(x)}{v_\infty(x)}\,.
 $$
Since $\gz$ is an isolated point, it follows that $Y$ is the punctured space,  and $mu_\infty$ and $v_\infty$ touch each other only inside $Y$. Therefore,
our assumption on the validity of SCP with respect to any two positive global solutions \eqref{eq:eq_with_scp1}  implies that
$m v_\infty = u_\infty = M v_\infty$ in $Y$, and therefore $M=m$.
\end{proof}

\section{The nonisolated singularity case}
\label{sec:boundary}
Some of our results extend to the case of a nonisolated singular point $\zeta \in \pd \hat{X}$. More precisely,
we consider two cases (cf. Section~\ref{sect1}):
\begin{enumerate}
\item $X$ is a domain (which might be unbounded and nonsmooth) such that the singular
point $\zeta=0$ belongs to a $C^2$-portion of $\pd X$.
\item  $X$ is a cone near infinity, and  $\zeta = \infty$. More precisely, the intersection of $X$ with the exterior of some ball is an open connected truncated cone with a nonempty $C^2$ boundary.
\end{enumerate}

One of the main technical difficulties of the nonisolated singularity case is that the
dilated equation is defined on a cone $Y$ rather on the punctured space as in the isolated singularity case, and
the regularity of neither the origin nor $\infty$ with respect to the  $p\,$-Laplace equation on a cone is known (cf. Example~\ref{example3}).
Hence we are unable to extend Theorem~\ref{thm:weak} to this case. On the other hand, Theorem~\ref{thm:assume_scp} can be readily extended. To this end, we introduce the following notion
\begin{definition}\label{def:bpl}{\em 
Let $\Omega'$ be a bounded subdomain of a domain $\Omega\subset \mathbb{R}^d$, $\Omega'\subsetneqq \Omega$, and let $\Gamma$ a $C^{2}$-relatively open portion of $\pd \Omega'$.
Assume that the equation $Q'(w)=0$ admits a positive solution in $\Omega$. We say that {\bf the boundary point lemma is valid for the equation $Q'(w)=0$ and $\Gamma$} if for any two positive solutions $u$ and $v$ of the equation $Q'(u)=0$ in $\Omega'$ that vanish continuously on $\Gamma$, and satisfy
$u<v$ in a neighborhood  $B_\varepsilon(y)\cap \Omega'$ of $y\in \Gamma$, we have
    $$\pd_\nu  v(y)<\pd_\nu u(y),$$ where $\nu$ denotes the unit outward normal to $\pd \Omega'$.
 }
\end{definition}

For the validity of the boundary point lemma for quasilinear equations see for example
\cite{Tak,Ser} and the references therein.

The following result can be proved along the lines of the
proof of Theorem~~\ref{thm:assume_scp} (see, Section~\ref{sec:proofs}).

\begin{theorem}\label{thm:assume_scpB}
Assume that \eqref{eq:1} admits a positive solution, and suppose that $V$ has a Fuchsian singularity at $\zeta$. Assume further that
there is a sequence $R_n \to \zeta$ such that SCP and the boundary point lemma hold true with respect to any two global positive solutions of the
limiting dilated equation
\begin{equation}\label{eq:ldeB}
\resc^{\{R_n\}}(Q)(w) = 0 \qquad \mbox{in }\, Y
\end{equation}
that vanish on $Y\sm\{0\}$.
 Then
\begin{enumerate}
\item[(i)] $\zeta$ is a regular point of the Eq.~\eqref{eq:1}.
\item[(ii)] Equation \eqref{eq:1} admits a unique positive solution
of minimal growth in a neighborhood of $\bo \hat{X} \sm \{\zeta\}$.
\end{enumerate}
\end{theorem}
We conclude this section with an example where we consider the $p\,$-Laplace equation on a smooth cone. Note that although we show below that the positive Liouville theorem holds true, we are unable to prove the regularity of the singular points $\gz=0,\infty$.
\begin{example}\label{example3}{\em 
 Let $X\subset \mathbb{R}^d$ be a cone generated by a smooth subdomain $S$ of the unit sphere $S_1$ such that $\pd S\neq \emptyset$. Consider the equation
   \begin{equation}\label{eq:pl}
    -\Delta_p(u)=0 \qquad \mbox{ in } X.
      \end{equation}
It is proved in \cite{Tolks,PV} that \eqref{eq:pl} admits a unique regular  (respectively, singular) {\em separable} positive $p\,$-harmonic function $u_\infty$ (respectively, $u_0$) in $X$
of the form
 $$u_\infty(x) =|x|^{\beta_\infty}\phi_\infty({x}/{|x|}), \qquad \mbox{ (respectively, } u_0 = |x|^{\beta_0}\phi_0({x}/{|x|}) \mbox{)},$$
where $\beta_0<0<\beta_\infty$, and $u_\infty$ (respectively, $u_0$) satisfies
$$u_\infty(x)\! =\!0 \; \forall x\!\in\! \partial X, \; \mbox{ (respectively, } u_0(x)\! =\!0\; \forall x\!\in\! \partial X\sm \{0\}, \;\lim_{x\to\infty}u_0(x)\!=\!0\mbox{)}.$$

Clearly, $u_\infty$ is a positive solution of  minimal growth in a neighborhood of $\bo \hat{X}\sm \{\infty\}$.
Moreover, if $p \geq d$, then $u_0$ is a positive solution of minimal growth in a neighborhood of $\bo \hat{X} \sm \{0\}$. Indeed, $|\grad u_0|$ is positive and $|\grad u_0(x)| \leq C |x|^{-1} |x|^{\beta_0}$.
For $k=1,\,2,\, \dots$ put $u_k := \chi_k u_0$, where $0\leq \chi_k\leq 1 $ is a smooth function such that
$$
\chi_k(x) = \left\{
          \begin{array}{ll}
    1 &  \quad  1<|x| < k, \\[2mm]
    0 &  \quad  0<|x|<1/2 \; \mbox{ or } |x| > 2k,
          \end{array}\right.
$$
and $|\grad \chi_k(x)| \leq C \frac{1}{k}$ if  $k<|x|<2k$.
We have $|\grad u_k(x)| \leq C k^{-1} |k|^{\beta_0}$ for $k < |x| < 2k$. Note that for any fixed ball $B$ in $B_1^*\cap X$ and large $k$ we have
$\int_B u_k(x)^p \,\mathrm{d}x=\mathrm{constant}>0$. Moreover,
the corresponding Lagrangian of Picone identity given by
$$
L(u_k,\,u) := \frac{1}{p} \left[|\grad u_k|^p + (p-1) \frac{u_k^p}{u^p}|\grad u|^p
-p \frac{u_k^{p-1}}{u^{p-1}} \grad{u_k} \cdot \grad u |\grad u|^{p-2} \right]
$$
is zero in $(B_k\sm B_1) \cup B^*_{2k}$.
 Hence, 
\begin{multline*}
\int_{X\sm B_1} L(u_k,u_0)\, \mathrm{d}x =\int_{\{x\mid k<|x|<2k\}} L(u_k,u_0)\, \mathrm{d}x \\[2mm]
\leq \int_{\{x\mid k<|x|<2k\}} C k^{-p} k^{\beta_0 p} \, \mathrm{d}x  \leq C k^{-p} k^{\beta_0 p} k^d
\end{multline*}
converges to zero as $k \to \infty$ if $p \geq d$ (actually, it is enough to assume that $d-p+\beta_0p<0$). Therefore, Theorem~7.1 of \cite{Pin2}  implies that $u_0$ is a positive solution of minimal growth in a neighborhood of $\bo \hat{X} \sm \{0\}$.

Note that the set $S_{u_\infty}$ (respectively, $S_{u_0}$) of critical points of $u_\infty$ (respectively, $u_0$) is empty. Moreover, by Theorem~2.15 of \cite{Lewis1}, the boundary point lemma is valid for the $p\,$-Laplacian on  $\pd X\sm \{0\}$. It follows as in the proof of Theorem~\ref{thm:sym} that $\infty$ (respectively, $0$) is a regular point with respect to $u_\infty$ and $u_0$, and that $u_\infty$ (respectively, $u_0$ if $p \geq d$) is the unique positive $p\,$-harmonic function in $X$ of minimal growth in  $\pd\hat{X}\sm \{\infty\}$ (respectively, $\pd\hat{X}\sm \{0\}$). In particular, it follows that the positive Liouville theorem for $p\,$-harmonic functions in $X$ that vanishes continuously on $\pd X$ holds true (without any Phragm\'en-Lindel\"of condition; cf.~Theorem~2.1.2 in \cite{Tolks}).
 }
\end{example}

\section{Further examples and remarks}\label{sec7}
In this section we present some examples and remarks which illustrate our results. We also present a new result concerning the existence of the limit of a positive solution $u$ at an isolated Fuchsian singular point in the case of a nonnegative potential $V$ (see Theorem~\ref{thm:nonnegativeV}).

The first example concerns positive solutions of \eqref{eq:hardy1} which appear  naturally in studying improved Hardy's inequality (see for example \cite{BFT} and the references therein).
\begin{example}\label{example2}{\em 
Let $X=\mathbb{R}^d\sm \{0\}$. Consider the equation
   \begin{equation}\label{eq:hardy1}
    -\Delta_p(u)-\lambda \frac{|u|^{p-2}u}{|x|^p} +V(x)|u|^{p-2}u=0 \qquad \mbox{ in } X.
      \end{equation}
where $\lambda\leq c_H=\left|\frac{p-d}{p}\right|^p$ is the Hardy constant, and $V$ satisfies the Fuchsian-type assumption
\begin{equation}\label{eq:FuchsCon3}
 |x|^p |V(x)| \leq C \qquad \mbox{for a. e. }  x\in X,
    \end{equation}
where $C$ is a positive constant. Note that the case $V=0$ is discussed in Example~\ref{example1}.
 Suppose further that $V$ has weak Fuchsian singularities both at $0$ and $\infty$.

 Remark~\ref{rem:8} implies that \eqref{eq:hardy1} admits a unique positive solution of minimal growth in  $\pd\hat{X}\sm \{0\}$, and a unique positive solution of minimal growth in $\pd\hat{X}\sm \{\infty\}$.
 }
 \end{example}
\begin{remark} \label{rem:VinLq}{\em 
Suppose that $\zeta=0 \in \pd \hat{X}$, $V \in
L^\infty_{\mathrm{loc}}(X)$ has a Fuchsian singularity at $0$
and $V \in L^q(B_1 \cap X)$ with $ q >d/p $. We claim that $V$ has a weak
Fuchsian singularity at $0$.

Indeed, let $\phi$ be a bounded function compactly supported in $B_r
\sm \{0\}$ for some $r>0$.  Then using H\"older's inequality we obtain
\begin{align*}
\left|\int_{X/R} R^p V(R x)\phi(x) \,\mathrm{d}x\right| & \leq R^{p-d}
\int_{X} |V(x)|\, |\phi(x/R)| \, \mathrm{d} x \\
& \leq ||\phi||_\infty R^{p-d} \int_{X \cap B_{Rr}} |V(x)| \, \mathrm{d} x\\
& \leq ||\phi||_\infty R^{p-d} |X \cap B_{Rr}|^{d -d/q}
||V||_{L^q(B_{Rr} \cap X)}.
\end{align*}
Therefore, there is a constant $C$ such that
\begin{multline*}
\left|\int_{X/R} R^p V(R x)\phi(x) \,\mathrm{d}x\right| \leq C
||\phi||_\infty R^{p-d} (Rr)^{d - d/q} ||V||_{L^q(B_{1} \cap X)}\\[2mm]
 \leq CR^{p-d/q}
 \underset{R\to 0}{\longrightarrow} 0.
\end{multline*}

 In the same
manner, one can prove that if $V \in L^\infty_{\mathrm{loc}}(X)$ has a
Fuchsian singularity at $\infty$ and
$V \in L^q(B^*_1 \cap X)$ with $ 1 \leq q <d/p $, then $V$ has a weak
Fuchsian singularity at $\infty$.
 }
\end{remark}
\begin{remark}\label{rem:2}{\em 
   The main results of the present paper hold true if instead of \eqref{eq:FuchsCon} one assumes that there exist $0<a<1<b<\infty$, and $R_n\to\zeta$ such that
   \begin{equation}\label{eq:FuchsConess}
 |x|^p |V(x)| \leq C \qquad \mbox{for a. e. } x\in \mathcal{A}:=\bigcup_{n=1}^\infty\{x\in X\mid aR_n<|x|<bR_n\}.
    \end{equation} Such a set $\mathcal{A}$ is called  an {\bf essential set with respect to the singular point $\zeta$} (see \cite{Pin1}).
    }
   \end{remark}

We conclude this section with an application of Proposition~\ref{thm:dilatation} to the asymptotic behavior of positive solutions near
singularity for an equation with a nonnegative potential $V$. The result is new even for linear case
and partially answers Question~9.5 in \cite{Pin1}. Note that the proof applies also in the linear  nonselfadjoint case.
  \begin{theorem}\label{thm:nonnegativeV}
    Let $V \geq 0$ be a Fuchsian-type potential with an isolated singularity at $\zeta = 0$ or $\zeta = \infty$. Assume that
for any sequence $R_n \to \zeta$ the limiting dilated equation satisfies $\resc^{\{R_n\}}(Q)(u) = -\pl(u)$ in $\mathbb{R}^d\sm \{0\}$.

Let $u$ be a positive solution of (\ref{eq:1}) near
$\zeta$. Then
$$
\lim_{x \to \zeta} u(x) \quad \mbox{exists}.
$$
The limit might be infinite.
  \end{theorem}
\begin{proof}
Without loss of generality assume that $\gz=\infty$. For $R_n \to \zeta$ set
$$m_R:=\inf_{x \in S_{R}} u(x),\;\; m_n :=m_{R_n}; \qquad M_R := \sup_{x \in S_{R}} u(x),\;\; M_n :=M_{R_n}.$$ By the maximum principle (cf. Lemma \ref{lem:mon}),  $M_R$ is a monotone function of $R$ for $R$ large enough. Hence, $\lim_{R \to \zeta} M_R$ exists and equals $M$, where $0\leq M\leq \infty$. If $M$
is infinite (respectively,  zero), then by the uniform Harnack inequality,  $\lim_{n \to \infty} m_n $ exists and is infinite (respectively, zero), and the claim follows.

Consequently, suppose that $u \asymp 1$ near $\gz$. We want to prove that for any sequence $\{R_n\}$, $\lim_{n \to \infty} m_n$ exists and is
equal to $M$. Assume to the contrary that there is a sequence $\{R_n\}$
such that $m:=\lim_{n \to \infty} m_n < M$.
Then $u_n(x) := u(R_n x)$, $n\geq 1$, is a uniformly bounded family that has a subsequence that converges  locally uniformly to a positive solution $u_\infty$ of the limiting dilated  equation $D^{\{R_n\}}(Q) w = 0$ in the punctured space. By our assumptions $u_\infty$ is
a bounded $p\,$-harmonic function in the punctured space, and therefore $u_\infty=\mathrm{constant}$. On the other hand, by our choice of $\{R_n\}$, we have
$$
 \inf_{x \in S_1} u_\infty(x) = \lim_{n \to \infty} m_n = m, \quad \mbox{and }\quad
 \sup_{x \in S_1} u_\infty(x) = \lim_{n \to \infty} M_n = M,
$$
which is a contradiction.
\end{proof}

   \vskip 3mm

\appendix
\renewcommand\thesection{\Alph{section}}
\section{Behavior of $p\,$-harmonic functions near infinity}
\label{sec:asymp}
In this Appendix we prove Theorem~\ref{thm:asymp2} concerning the asymptotic behavior of positive $p\,$-harmonic functions near infinity for $p\geq d$. To this end, we use a modification of the classical Kelvin transform $u(x)\mapsto  |x|^{2-d} u(x/|x|^{2})$ that preserves classical harmonic functions.
\begin{definition}\label{def:Kelvin}{\em 
For $x \in \mathbb{R}^d$ we denote by $\tilde{x}: = {x}/{|x|^2}$  the {\bf inverse point with respect to the unit ball} $B_1$. Let $u$ be a function defined either on the punctured unit ball or on $B_1^*$, the exterior of the unit ball. The {\bf modified Kelvin transform} of $u$ is defined by
$$
 K[u](x): = u({x}/{|x|^2})= u(\tilde{x}).
$$
 }
 \end{definition}
For $p=d$, the modified Kelvin transform is a conformal transformation, therefore, in this case Theorem~\ref{thm:asymp2} follows from Theorem~\ref{thm:asymp}.
The proof of Theorem~\ref{thm:asymp2} for $p > d$ consists of a sequence of lemmas. First we use the modified Kelvin transform to analyze the behavior of positive $p\,$-harmonic functions near infinity  from the behavior of positive solutions of a weighted $p\,$-Laplace equation near the origin. Then, following closely the proof of Serrin in \cite{Ser1}, we find the asymptotic of positive singular solutions near an isolated singular point.
\begin{lemma}\label{eq:plw}
Let $\beta:=2(p-d)$. Suppose that $u$ is a solution of the equation $-\pl (u) =0$ in a neighborhood of infinity (respectively, in a punctured neighborhood of the origin), and let $v:=K[u]$. Then
$v$ is a solution of the equation
\begin{equation}
-\diver(A [ v]):=-\diver( |x|^\beta |\grad v|^{p-2} \grad v) = 0
\end{equation}
in a punctured neighborhood of the origin (respectively, in a neighborhood of infinity).

In particular, any bounded solution of the equation $-\diver(A [ v])=0$ in the punctured space is a constant.
\end{lemma}
\begin{proof}
An elementary computation shows that
\newcommand{\gsu}{\tilde{\grad} u(\tilde x)}
\newcommand{\gsui}[1]{\tilde{\grad}_#1 u(\tilde x)}
\newcommand{\gst}[1]{\tilde{\grad}_#1}
\begin{eqnarray}
\grad_i v(x) &=& \gst{j} u(\tilde x) \left(\frac{\delta_{ij}}{|x|^2} - 2\frac{x_i x_j}{|x|^4} \right) = \gsui{i} |\tilde x|^2 - 2 \gsui{j}  \tilde x_j \tilde x_i,
\nonumber \\
|\grad v(x)|^2 &=& |\gsu|^2 \frac{1}{|x|^4} = |\gsu|^2 |\tilde x|^4 \nonumber,
\end{eqnarray}
here and below we sum over repeated indices and use $\tilde{\grad}$ to denote
gradient with respect to $\tilde x$.
Accordingly,
\begin{align*}
\diver( A[v] )
&= \diver \left[|x|^{2 - 2d} |\gsu|^{p-2}\left( \gsu -2 \gsui{j} x_j \frac{x}{|x|^2}\right) \right] \\
&= \diver \left[ |\tilde x|^{2d-2} |\gsu|^{p-2}\left( \gsu -2 \gsui{j} \frac{\tilde x_j}{|\tilde x|^2} \tilde x \right) \right]  \\
&= \gst{i} \!\left[\!|\tilde x|^{2d\!-\!2} |\gsu|^{p\!-\!2}\!\left(\! \gsui{j} \!-\!2 \gsui{k} \frac{\tilde x_k}{|\tilde x|^2} \tilde x_j\right)\! \right]\!\! \left(\delta_{ij} |\tilde x|^2 \!-\! 2 \tilde x_i \tilde x_j\!\right).
\end{align*}
Expanding the gradient in the last line yields
\begin{multline*}
\diver(A[v]) = \\ \left[\! \!(2d\! -\!2) |\tilde x|^{2d\! -\! 4} \tilde x_i |\gsu|^{p\!-\!2} \gsui{j}
\!-\! 2 (2d\!-\!4) |\tilde x|^{2d\! -\! 6} \tilde x_i |\gsu|^{p\!-\!2} \gsui{k} \tilde x_k \tilde{x}_j \right.\\
+|\tilde x|^{2d-2} \gst{i} \left(|\gsu|^{p-2} \gsui{j}\right) - 2 |\tilde{x}|^{2d-4} \gst{i} \left(|\gsu|^{p-2} \gsui{k}\right) \tilde{x}_k \tilde{x}_j \\
\left. -2 |\tilde x|^{2d -4} |\gsu|^{p-2} \gsui{k} \left(\delta_{ki} \tilde x_j +
\tilde x_k \delta_{ij}\right) \right] (\delta_{ij}|\tilde x|^2 - 2 \tilde x_i \tilde x_j).
\end{multline*}
By collecting the terms of the first and the last lines of the right hand side of the latter equation, we get
\begin{align}\label{eq:divA}
\diver(A[v]) =& |\tilde{x}|^{2d-2} |\gsu|^{p-2} \tilde{x}_i \gsui{i}
(-(2d -2) + 2(2d-4) + 6 -2d) \nonumber \\
&+|\tilde x|^{2d-2} \gst{i} \left(|\gsu|^{p-2} \gsui{j}\right)(\delta_{ij}|\tilde x|^2 - 2 \tilde x_i \tilde x_j) \\
&- 2 |\tilde{x}|^{2d-4} \gst{i} \left(|\gsu|^{p-2} \gsui{k}\right) \tilde{x}_k \tilde{x}_j(\delta_{ij}|\tilde x|^2 - 2 \tilde x_i \tilde x_j).\nonumber
\end{align}
Note that the right hand side of the first line of \eqref{eq:divA} equals zero, while the last two lines of \eqref{eq:divA} give us $$
\diver(A[v]) = |\tilde x|^{2d-2} \gst{i} \left(|\gsu|^{p-2} \gsui{j}\right) \delta_{ij}|\tilde x|^2 = |\tilde x|^{2d} \tilde \laplace_p u(\tilde x),
$$
which is equal to zero by our assumption.

Since the Liouville theorem holds true for bounded $p\,$-harmonic functions in $\mathbb{R}^d$, it follows from the first part of the proof that any bounded solution of the equation $-\diver(A [ v])=0$ in the punctured space is a constant.
\end{proof}
Denote $\alpha:=(d-p)/(p-1)$.  It is well known that $|x|^{-\alpha}$ is
a positive radial solution of the $p\,$-Laplace equation in the punctured space. By Lemma~\ref{eq:plw},  the function $|x|^\alpha$
is a solution of the weighted equation $-\diver(A [ v])=0$ in the punctured space. It is useful to keep in mind that for $p>d$, $\beta$ is positive and $\alpha$ is negative.

The following lemma claims that $\zeta=\infty$ is a regular point with respect to the constant function and the $p\,$-Laplace equation.
\begin{lemma}\label{eq:plw1}
Let $u$ be a positive solution of the equation $-\pl (u) =0$ in a neighborhood of infinity, where $1<p<\infty$.   If $u$ does not admit a finite limit as $x\to\infty$, then
$$\lim_{x \to \infty} u(x) = \infty.$$
 Moreover, if $p\geq d$, then $\lim_{x \to \infty} u(x)\neq 0$.
\end{lemma}
\begin{proof}
By Lemma~\ref{lem:mon} (with $v=\mathbf{1}$),  for $r$ large enough, the functions
$$
m_r := \inf_{x\in S_r} u(x),\, \quad M_r:=\sup_{x\in S_r} u(x)
$$
are monotone. Suppose that $m= \lim_{r \to \infty} m_r$ is finite.
Then for any $\varepsilon > 0$ the function $u - m +\varepsilon$
is a positive $p\,$-harmonic function in some neighborhood
of infinity. By the uniform Harnack inequality, for $r$ large we have
$$M_r - m + \varepsilon < C(m_r - m +\varepsilon)$$
and by taking $r \to \infty$ we get
$$0\leq M - m \leq C\varepsilon.$$
Hence $M=m<\infty$, and $\lim_{|x| \to \infty}u(x)=m=M<\infty$. Therefore,  $u$ admits a finite limit as $x\to\infty$.

We note that for $p\geq d$ this finite limit is in fact positive.  Indeed, in this case, the equation $-\pl (u) =0$ in $\mathbb{R}^d$ is critical  \cite{Pin2}, and hence,  the positive constant function is its ground state. In particular,  the constant function has minimal growth at infinity. So, for $p\geq d$, there is no positive $p\,$-harmonic function in an exterior domain that tends to zero at infinity.
\end{proof}
\begin{remark}\label{rem:pgeqd}{\em 
Using the ``fundamental solution" of the $p\,$-Laplacian in $\mathbb{R}^d$
 \begin{equation}\label{eq:fundemtal}
    \mu_p(x):=\begin{cases}
C(d,p)|x|^{(p-d)/(p-1)}    & \text{ if } p\neq d, \\[2mm]
    C(d)\log(|x|) & \text{ if } p=d,
  \end{cases}
 \end{equation}
  and a simple comparison argument, it can be shown that a positive solution of the equation $-\pl (u) =0$ in a neighborhood of infinity satisfying $\lim_{x \to \infty} u(x) = \infty$ exists if and only if $p\geq d$. For an extension of this result see \cite{MY3}.
 }
\end{remark}
\begin{lemma}\label{thm:la}
For $\!p>\!d$, let $u$ be a positive solution of the equation $-\pl (u) \!=\!0$ in a neighborhood of infinity, satisfying $\lim_{x\to\infty} u(x)=\infty$. Fix $R>0$ and  $c>0$ such that $v_c(x):=K[u](x)-c$ is positive near the origin and negative on $S_R$.  For any $\theta\in C^1_0(B_R)$ which is identically $1$ near the origin we have
$$
\int_{B_R} \nabla \theta \cdot A[v_c] \,\mathrm{d}x= k,
$$
where $k$ is a positive constant independent of $\theta$.
\end{lemma}
\begin{proof}
 Note that the difference of any two such $\theta$ has a compact support in $B_R\sm \overline{B_\varepsilon}$ for some $\varepsilon>0$. Since  $v_c$ is a solution of the equation $-\diver(A [u])=0$ in $B_R\sm \overline{B_\varepsilon}$, it follows that
 $$\int_{B_R} \nabla \theta \cdot A[v_c] \,\mathrm{d}x= \mathrm{constant}=k .$$ Therefore, it remains to prove that $k>0$. Let
$$
\theta(x) := \left\{
\begin{array}{lr}
1 & v_c(x) \geq 1, \\
v_c(x) & 0<v_c(x) < 1, \\
0 & v_c(x) \leq 0.
\end{array} \right.
$$
Since $v_c$ is not a constant it follows that
$$
k=\int_{B_R} \nabla \theta \cdot A[v_c] \,\mathrm{d}x =
\int_{\{x\in B_R\mid 0<v_c(x)<1\}} |x|^\beta |\grad v_c|^p  \,\mathrm{d}x  >0.
$$
\end{proof}

The following lemma can be found in \cite{Martio} (see Theorem 7.41 therein). However, we
include the proof for the completeness.
\begin{lemma}\label{thm:la1}
Let $p>d$, and let $v_c(x)$ be the solution in Lemma~\ref{thm:la}. Then there exists $\varepsilon>0$ such that \begin{equation}\label{eq:la1}
v_c(x) \asymp |x|^\alpha \qquad \mbox{ in } B_\varepsilon\sm \{0\}.
\end{equation}
\end{lemma}
\begin{proof} We denote by $C$ a generic positive constant whose
value might change from line to line but depends only on $p$ and $d$.
For $0<r<R$ let
$$
m_r= \inf_{x\in S_r} v_c(x), \qquad M_r=\sup_{x\in S_r} v_c(x).
$$
In light of Lemma~\ref{lem:mon}, and since $\lim_{x\to 0}v_c(x)=\infty$, we may assume that $m_r,\,M_r$ are nondecreasing as $r \to 0$.
It suffices to prove that there exist positive constants $C$ and $ C_1$  depending only on $p$ and $d$ such that
$$m_r \leq C r^\alpha \leq C_1M_r \qquad \mbox{ for all }  0<r<r_0,$$
and then estimate \eqref{eq:la1} will follow using
the uniform Harnack inequality (Lemma~\ref{lem:BHI}).

For $\mu>0$ let
$$
\theta_\mu (x) := \left\{
\begin{array}{lr}
\mu & \quad v_c(x) \geq \mu, \\
v_c(x) &  \quad 0<v_c(x) < \mu, \\
0 &  \quad v_c(x) \leq 0.
\end{array} \right.
$$
Note that $\theta_\mu(x) = \mu$ near the origin, therefore, by Lemma~\ref{thm:la} we have
\begin{equation}\label{eq:kmu}
k \mu =\int_{B_R} \nabla \theta_\mu \cdot A[v_c] \,\mathrm{d}x.
\end{equation}

\noindent \textbf{Upper bound:}
By \eqref{eq:kmu} we have
$$
km_r = \int_{B_R} \nabla \theta_{m_r} \cdot A[v_c] \,\mathrm{d}x = C\int_{B_R} |x|^\beta |\grad \theta_{m_r}|^p \,\mathrm{d}x \geq C_1m_r^p\, \mathrm{cap}_{p,\,\beta}(B_{r,R}),
$$
where $\mathrm{cap}_{p,\,\beta}(B_{r,R})$ is the (variational) weighted $p\,$-capacity of the ball $B_r$ in $B_R$ with respect to measure $|x|^\beta$
(see, \cite{Martio}, Chapter 2). Explicit calculation \cite{Martio} (see Example~2.22 therein), shows that
$$
\mathrm{cap}_{p,\,\beta}(B_{r,R}) = C\left( r^{(p-d-\beta)/(p-1)} - R^{(p-d-\beta)/(1-p)}
\right)^{1-p}.
$$
Thus,
$$
k m_r^{1-p} \geq C \left( r^{(p-d-\beta)/(p-1)} - R^{(p-d-\beta)/(1-p)}
\right)^{1-p}.
$$
Note that $(p-d-\beta)/(p-1)=\alpha$, therefore,
$$
m_r \leq C\left( r^\alpha - R^\alpha \right) \leq C r^\alpha.
$$
\vskip 3mm
\noindent \textbf{Lower bound:}
For $0<r <R$, let
$$
\xi_r(x):=\left\{\begin{array}{lr}
    1 &  \quad |x| < r,\\[3mm]
    \dfrac{|x|^\alpha -R^\alpha}{r^\alpha - R^\alpha} &  \quad r \leq |x| \leq R, \\[3mm]
    0&  \quad |x|>R.
    \end{array} \right.
$$
 Using Lemma~\ref{thm:la} and H\"older's inequality, we obtain
\begin{equation}\label{eq:14}
k\!=\!\!\int_{B_R} \!\!\!\!\nabla \xi_r \cdot A[v_c] \mathrm{d}x \leq \!\left(\!\int_{B_R\sm B_r} \!\!\!|\nabla \xi_r|^p |x|^\beta \mathrm{d}x\right)^{\!1/p}\!\!\!
\left(\! \int_{B_R\sm B_r}\!\!\! |\grad v_c|^p |x|^\beta \mathrm{d}x \right)^{\!\!(p-1)/p}\!\!\!.
\end{equation}
For the first term in the right hand side of \eqref{eq:14} we have
\begin{multline*}
\int_{B_R\sm B_r} |\nabla \xi_r|^p |x|^\beta \,\mathrm{d}x
 = \frac{C}{(r^\alpha -R^\alpha)^p}
\left( r^{(\alpha -1)p + \beta + d} - R^{(\alpha -1)p + \beta + d} \right)\\
=\frac{C}{(r^\alpha - R^\alpha)^{p-1}}\;,
\end{multline*}
where we used the equality $(\alpha -1)p + \beta + d = \alpha$.
Consequently, for $r$ small we have
\begin{equation}\label{eq:1rhs}
\int_{B_R\sm B_r} |\nabla \xi_r|^p |x|^\beta \,\mathrm{d}x
 \leq C r^{-\alpha (p-1)}.
\end{equation}
To estimate the second term of (\ref{eq:14}), we note that $v_c = \theta_{M_r}$
in $\{0 \leq v_c \leq M_r\} \supset B_R \sm B_r$, and consequently
\begin{equation}\label{eq:2rhs}
\int_{B_R \sm B_r}\!\!\! |\grad v_c|^p |x|^\beta  \mathrm{d}x
\leq \!\!\int_{0 \leq v_c \leq M_r}\!\!\! |\grad v_c|^p |x|^\beta  \mathrm{d}x
= \!\!\int_{B_R}\!\!\! \grad \theta_{M_r} \cdot A[v_c] \, \mathrm{d}x = k M_r.
\end{equation}
Combining \eqref{eq:14}, \eqref{eq:1rhs} and \eqref{eq:2rhs} we get
\begin{equation*}
k = \int_{B_R} \nabla \theta_r \cdot A[v_c] \,\mathrm{d}x \leq C  r^{\alpha(1-p)/p}\,M_r^{(p-1)/p}
\end{equation*}
and $ r^\alpha \leq C M_r$ follows.
\end{proof}
With the work of the preceding lemmas available, we can finish the proof of Theorem~\ref{thm:asymp2} concerning the asymptotic behavior of positive singular $p\,$-harmonic functions near infinity for $p > d$.
\begin{proof}[{End of the proof of Theorem~\ref{thm:asymp2}}]
 In light of Lemma~\ref{eq:plw}, Lemma~\ref{eq:plw1}, and Lemma~\ref{thm:la1}, we need only to show that $v(x) \asymp |x|^\alpha$ in $B_\varepsilon \sm \{0\}$ implies that $v(x) \sim |x|^\alpha$ as $x\to 0$.
We use V\'{e}ron's scaling method \cite{Pin2}.

For $0<\sigma <1$, we denote $w_\sigma(x) := v(\sigma x)/\sigma^\alpha$. Then $\{w_\sigma\}_{0<\sigma<1}$  is
a locally bounded family, and $w_\sigma(x)\asymp |x|^\alpha$ in $B_{c/\sigma}\sm \{0\}$ for some $c>0$. Consequently,  there is a subsequence $\sigma_n\to 0$ such that $\{w_{\sigma_n}\}$ converges locally uniformly in $\mathbb{R}^d\sm \{0\}$ to
$$
W(x) := \lim_{n \to \infty} w_{\sigma_n}(x)
$$
which is a positive solution of the equation
$$
-\diver(A[u]) = 0 \qquad \mbox{ in } \mathbb{R}^d\sm \{0\}.
$$
Clearly, $m |x|^\alpha \leq W(x) \leq M |x|^\alpha$, where
$$
m:=\lim_{r\to 0}m_r=\lim_{r\to 0} \inf_{x\in S_r} \frac{v(x)}{r^\alpha}, \qquad M:=\lim_{r\to 0}M_r=\lim_{r\to 0}\sup_{x\in S_r} \frac{v(x)}{r^\alpha},
$$
and the existence of the limits above follows from Lemma \ref{lem:mon}.

We claim that for any $R > 0$ we have
$$
\inf_{x \in S_R} \frac{W(x)}{|x|^\alpha} = m ,\qquad
\sup_{x \in S_R} \frac{W(x)}{|x|^\alpha} = M.
$$
Indeed,
\begin{multline*}
\inf_{x \in S_R} \frac{W(x)}{|x|^\alpha}
=\inf_{x \in S_R} \lim_{n \to \infty} \frac{w_{\sigma_n}(x)}{|x|^\alpha} =
 \lim_{n \to \infty} \inf_{x \in S_R}  \frac{w_{\sigma_n}(x)}{|x|^\alpha}\\
=
 \lim_{n \to \infty}
\inf_{x \in S_R} \frac{v(\sigma_n x)}{(\sigma_n R)^\alpha} = \lim_{n \to \infty}
\inf_{x \in S_{\sigma_n R}} \frac{v(x)}{|x|^\alpha} = \lim_{n \to \infty}m_{\sigma_n R} =m,
    \end{multline*}
where the interchanging of the order of the two limiting processes above is justified due to the local uniform convergence of the sequence $\{w_{\sigma_n}(x)/|x|^\alpha\}$.
Similarly, we obtain $\sup_{x \in S_R} \frac{W(x)}{|x|^\alpha} =
M$.

Since $|x|^\alpha$ is a positive solution of the equation $-\diver(A[u]) = 0$ in  $\mathbb{R}^d\sm \{0\}$ which does not have any
critical point, it
follows from Theorem \ref{thm:SCP} that $ m|x|^\alpha = W(x) = M |x|^\alpha$,
and therefore $m=M$.
\end{proof}

\begin{center}{\bf Acknowledgments} \end{center}
 M.~F. and Y.~P. acknowledge the support of the Israel Science
Foundation (grant No. 587/07) founded by the Israel Academy of
Sciences and Humanities. M.~F. acknowledges also the support of the Israel Science
Foundation (grant No. 419/07) founded by the Israel Academy of
Sciences and Humanities. M.~F. was also partially supported by a fellowship
of the UNESCO fund.


\end{document}